\documentclass[10pt, a4paper]{article}

\usepackage{fullpage}
\usepackage{amsmath}
\usepackage{amsthm}
\usepackage{amsfonts}
\usepackage{tikz-cd}
\usepackage{chngcntr}
\usepackage[utf8x]{inputenc}
\usepackage{enumerate}
\usepackage[english]{babel}
\usepackage{amssymb}
\usepackage{esint}
\usepackage{hyperref}
\usepackage{verbatim}
\usepackage{mathtools}

\numberwithin{equation}{section}

\theoremstyle{plain}
\newtheorem{thm}[equation]{Theorem}
\newtheorem{lemma}[equation]{Lemma}
\newtheorem{cor}[equation]{Corollary}
\newtheorem{prop}[equation]{Proposition}
\theoremstyle{definition}
\newtheorem{dfn}[equation]{Definition}
\theoremstyle{remark}

\newtheorem{rmk}[equation]{Remark}

\newcommand{\circleland}{ 
	\mathbin{
		\mathchoice
		{\buildcircleland{\displaystyle}}
		{\buildcircleland{\textstyle}}
		{\buildcircleland{\scriptstyle}}
		{\buildcircleland{\scriptscriptstyle}}
	} 
}

\newcommand\buildcircleland[1]{%
	\begin{tikzpicture}[baseline=(X.base), inner sep=0, outer sep=0]
	\node[draw,circle] (X)  {$#1\land$};
	\end{tikzpicture}%
}

\title{On the rigidity of harmonic-Ricci solitons}
\author{Andrea Anselli\thanks{\url{andrea.anselli@unimi.it}}\\ Dipartimento di Matematica \textquotedblleft Federigo Enriques\textquotedblright, Universit\'{a} degli Studi di Milano,\\
	Via Cesare Saldini 50, 20133 Milano, Italia.\thanks{No longer affiliated}}
\date{}

\begin{document}
	
	\maketitle
	
	\begin{abstract}
		In this paper we introduce the notion of rigidity for harmonic-Ricci solitons and we provide some characterizations of rigidity, generalizing some known results for Ricci solitons. In the compact case we are able to deal with not necessarily gradient solitons while, in the complete non-compact case, we restrict our attention to steady and shrinking gradient solitons. We show that the rigidity can be traced back to the vanishing of certain modified curvature tensors that take into account the geometry a Riemannian manifold equipped with a smooth map $\varphi$, called $\varphi$-curvatures, which are a natural generalization of the standard curvature tensors in the setting of harmonic-Ricci solitons.
	\end{abstract}
	
	\tableofcontents
	
	\section{Introduction}
	
	Ricci solitons, i.e., self-similar solutions of the Ricci flow, had been intensively studied since the pioneering work of R. S. Hamilton \cite{H}. They have been subject of the studies of an incredible number of scientists: we recommend, for instance, the survey paper \cite{C} or Chapter 8 of \cite{AMR} and all the references therein. More recently, B. List and R. M\"{u}ller, see \cite{L} and \cite{M}, respectively, combined the Ricci flow for a metric with the heat flow for a map obtaining the so called harmonic-Ricci flow. Self-similar solutions of the harmonic-Ricci flow are called harmonic-Ricci solitons and quite recently started attracting some attention; see, for instance, \cite{W}, \cite{YS} and \cite{YZ}.
	
	Recall that a Riemannian manifold $(M,g)$ is called harmonic-Ricci soliton (with respect to the positive constant $\alpha$, the smooth map $\varphi:M\to N$, where the target $(N,\langle\,,\,\rangle_N)$ is a Riemannian manifold, the vector field $X$ and $\lambda\in \mathbb{R}$) if the following hold
	\begin{equation}\label{harm ricci soliton con campo vettoriale}
	\begin{dcases}
	\mbox{Ric}-\alpha \varphi^*\langle\,,\,\rangle_N+\frac{1}{2}\mathcal{L}_Xg=\lambda g\\
	\tau(\varphi)=d\varphi(X),
	\end{dcases}
	\end{equation}
	where $\tau(\varphi)$ denotes the tension field of $\varphi$ and $\mathcal{L}_Xg$ the Lie derivative of the metric $g$ in the direction of $X$. The harmonic-Ricci soliton is called shrinking, steady or expanding if, respectively, $\lambda>0$, $\lambda=0$ or $\lambda<0$, and it is called gradient if the vector field can be replaced with the gradient of a smooth function $f$ on $M$, called potential function. For gradient harmonic-Ricci solitons \eqref{harm ricci soliton con campo vettoriale} reads
	\begin{equation*}
	\begin{dcases}
	\mbox{Ric}-\alpha \varphi^*\langle\,,\,\rangle_N+\mbox{Hess}(f)=\lambda g\\
	\tau(\varphi)=d\varphi(\nabla f).
	\end{dcases}
	\end{equation*}
	
	While not much is yet known about harmonic-Ricci solitons, many interesting classification results are available for Ricci solitons. A key point in the study of Ricci solitons is the understanding of the central role of Einstein metrics. Indeed, any Einstein manifold endowed with a Killing vector field give rise to a trivial Ricci soliton. Starting from Einstein manifolds it is possible to build more general and less trivial examples of gradient Ricci solitons, the so called rigid gradient Ricci solitons introduced by P. Petersen and W. Wylie in \cite{PW}. A gradient Ricci soliton is said to be rigid if it is isometric to a quotient of the Riemannian product $L\times \mathbb{R}^k$, where $L$ is an Einstein manifold and $f(x)=\frac{\lambda}{2}|x|^2$ on the Euclidean factor. 
	
	Actually, in the compact case nothing changes: a compact gradient Ricci soliton is rigid if and only if it is Einstein and the potential function is constant and, moreover, rigidity is equivalent to the constancy of the scalar curvature, see for instance the work \cite{ELM} of M. Eminenti, G. La Nave and C. Mantegazza. In the complete non-compact case we have an analogous situation for steady gradient Ricci solitons, since they are rigid precisely when the scalar curvature is constant, see Proposition 3.2 of \cite{PW}. The constancy of the scalar curvature is not sufficient to characterize rigidity for complete non-steady gradient Ricci solitons, one needs also that the soliton is radially flat, i.e, that $R(\cdot,\nabla f)\nabla f=0$, as showed in Theorem 1.2 of \cite{PW}. Actually A. Naber, in Theorem 2 of \cite{N}, was the first one who shows the rigidity of a certain class of gradient Ricci solitons.
	
	Later on M. Fernández-López and E. García-Río, in \cite{FG}, and O. Munteanu and N. Sesum, in \cite{MS}, showed that rigidity for complete gradient Ricci solitons is equivalent to having harmonic Weyl curvature, relying on the results of \cite{PW} mentioned above. In Theorem 2.1 of \cite{FG} the authors showed that a compact Ricci soliton, not necessarily gradient, is Einstein, i.e., is rigid, if and only if it has harmonic-Weyl tensor.
	
	The results of \cite{FG} and \cite{MS} reveals fundamental in the work \cite{CC} of H. D. Cao and Q. Chen, where the authors studied complete Bach-flat gradient Ricci soliton obtaining, among other things, their rigidity (see Theorem 1.1 and Theorem 1.2). By an accurate analysis of the geometry of the regular level sets of the potential function they obtained that the vanishing of the Bach tensor implies that the Weyl tensor is harmonic, hence they reduce the proof to the results of \cite{FG} and \cite{MS}.
	
	The first aim of this article is to introduce the notion of rigidity for gradient harmonic-Ricci soliton. The key role played by Einstein manifolds for Ricci solitons, in the setting of harmonic-Ricci solitons is occupied by harmonic-Einstein manifolds. Recall that a Riemannian manifold $(M,g)$ of dimension $m\geq 2$ is called harmonic-Einstein (with respect to the positive constant $\alpha$, the smooth map $\varphi:M\to N$, where the target $(N,\langle\,,\,\rangle_N)$ is a Riemannian manifold and $\lambda\in \mathbb{R}$) if the following hold
	\begin{equation*}
	\begin{dcases}
	\mbox{Ric}-\alpha \varphi^*\langle\,,\,\rangle_N=\lambda g\\
	\tau(\varphi)=0.
	\end{dcases}
	\end{equation*}
	The author, in \cite{ACR} and \cite{A}, studied the geometric properties of some curvature tensors on a Riemannian manifold that take in account the geometry of a smooth manifold equipped with a smooth map $\varphi$, the so called $\varphi$-curvatures. It is important to observe that for a harmonic-Einstein manifold the $\varphi$-scalar curvature $S^{\varphi}:=S-\alpha|d\varphi|^2$, where $|d\varphi|^2$ is the Hilbert-Schimdt norm of the differential of $\varphi$, being the trace of the $\varphi$-Ricci tensor $\mbox{Ric}^{\varphi}:=\mbox{Ric}-\alpha \varphi^*\langle\,,\,\rangle_N$, is constant on the whole $M$. Starting from the $\varphi$-Ricci tensor one may define, among the other $\varphi$-curvatures, the $\varphi$-Cotton, the $\varphi$-Weyl and the $\varphi$-Bach tensors, see \hyperref[section phi curv]{Section \ref*{section phi curv}} for the details.
	
	We say that a gradient harmonic-Ricci soliton with respect to $\alpha$, $\varphi:M\to N$ smooth and $\lambda$ and with potential function $f$ is rigid if it is isometric to a quotient of the Riemannian product $L\times \mathbb{R}^k$, where $L$ is a harmonic-Einstein manifold with respect to $\alpha$, $\psi:M\to N$ smooth and $\lambda$, the potential $f=\frac{\lambda}{2}|x|^2$ on the Euclidean factor and the map $\varphi=\psi$ on the harmonic-Einstein factor, see \hyperref[def rigidita]{Definition \ref*{def rigidita}} below for a more precise formulation.
	
	A compact harmonic-Ricci soliton (not necessarily gradient) is called rigid if it is harmonic-Einstein. We will see that for compact harmonic-Ricci soliton rigidity is equivalent to having constant $\varphi$-scalar curvature. The main result for compact solitons is \hyperref[thm phi cotton flat compact ricci solitons are rigid]{Theorem \ref*{thm phi cotton flat compact ricci solitons are rigid}} that is a generalization, at least for dimension $m\geq 3$, of Theorem 2.1 of \cite{FG}. At first one may think that the natural generalization of the hypothesis of having harmonic-Weyl tensor is to have harmonic $\varphi$-Weyl tensor, but this is not true. For manifolds of dimension $m\geq 4$ having harmonic Weyl tensor is equivalent to the vanishing of the Cotton tensor. What is true then is that, as stated in \hyperref[thm phi cotton flat compact ricci solitons are rigid]{Theorem \ref*{thm phi cotton flat compact ricci solitons are rigid}}, $\varphi$-Cotton flat compact harmonic-Ricci soliton of dimension $m\geq 3$ are rigid.
	
	Moving on to the complete case, in \hyperref[prop caratt caso steady]{Proposition \ref*{prop caratt caso steady}} we show that the constancy of the $\varphi$-scalar curvature is necessary and sufficient for the rigidity of complete gradient steady Ricci-harmonic soliton, extending Proposition 3.2 of \cite{PW} in our setting. For complete non-steady gradient harmonic-Ricci solitons we characterize rigidity via the condition of having parallel $\varphi$-Ricci tensor, see \hyperref[thm rigidity]{Theorem \ref*{thm rigidity}}. This result is similar to Theorem 1.2 of \cite{PW}, although our hypothesis is slightly stronger, see \hyperref[rmk ipotesi petersen wylei]{Remark \ref*{rmk ipotesi petersen wylei}}. Nevertheless, relying on \hyperref[thm rigidity]{Theorem \ref*{thm rigidity}}, we are able to extend the results \cite{FG} and \cite{MS} regarding complete non-compact shrinking soliton. Indeed, in \hyperref[thm caso non compatto con phi cotton nullo]{Theorem \ref*{thm caso non compatto con phi cotton nullo}}, we show that a complete non-compact gradient shrinking harmonic-Ricci soliton with vanishing $\varphi$-Cotton tensor satisfies $\nabla \mbox{Ric}^{\varphi}=0$ and then, in \hyperref[cor se c phi zero allora rigidita]{Corollary \ref*{cor se c phi zero allora rigidita}}, we characterize the rigidity of complete non-compact gradient shrinking harmonic-Ricci soliton via the vanishing of the $\varphi$-Cotton tensor.
	
	Our final aim is the generalization of the results of \cite{CC} mentioned above. What we obtained is that the vanishing of the $\varphi$-Bach tensor characterize rigidity for complete gradient shrinking harmonic-Ricci solitons (actually, in dimension $m\neq  4$, we require the vanishing of the totally traceless part of the $\varphi$-Bach tensor, see \hyperref[rmk totally traceless part]{Remark \ref*{rmk totally traceless part}} for further comments on this assumption). Recall that Bach flat metrics for compact four dimensional Riemannian manifold are critical points of the conformal invariant functional
	\begin{equation}\label{funct int di norma di W}
		g\mapsto \int_M|W_g|^2\mu_g,
	\end{equation}
	where $g$ is a Riemannian metric, $W_g$ is the Weyl tensor and $\mu_g$ is the Riemannian volume element. The study of Bach flat metrics is interesting in view of their role in General Relativity, see \cite{B}. In \cite{A20} the author studied critical points of the conformal invariant functional $\mathcal{S}_2$, that is a more general version of the functional \eqref{funct int di norma di W} that does not depend only on a Riemannian metric $g$ but also on a smooth map $\varphi$, on a compact four dimensional manifold. Its critical points $(g,\varphi)$ are characterized by the validity of
	\begin{equation}\label{pair equation intro}
		B^{\varphi}=0, \quad J=0,
	\end{equation}
	where $B^{\varphi}$ is the $\varphi$-Bach tensor and $J$ is the section of the pullback bundle $\varphi^{-1}TN$ defined in \eqref{def di J} below. The main interest of \eqref{pair equation intro} resides in their applications in General relativity, as pointed out in \cite{A20}. See \hyperref[rmk bach flat gen rel]{Remark \ref*{rmk bach flat gen rel}} for more details and the definition of $\mathcal{S}_2$.
	
	The study of $\varphi$-Bach flat harmonic-Ricci solitons is a first attempt to find non-trivial examples (i.e, not harmonic-Einstein) of critical points of $\mathcal{S}_2$ and it has been the main motivation of our study. As mentioned above our main result says, essentially, that $\varphi$-Bach flat complete gradient shrinking harmonic-Ricci soliton are rigid: as usual in the compact case we are able to deal with harmonic-Ricci solitons that are not necessarily gradient, see \hyperref[cor bach flat harm ricci soliton sono harm einst]{Corollary \ref*{cor bach flat harm ricci soliton sono harm einst}}, while in the complete non-compact case we restricted our attention to gradient shrinking harmonic-Ricci solitons, see \hyperref[cor bach non comp flat harm ricci soliton sono harm einst]{Corollary \ref*{cor bach non comp flat harm ricci soliton sono harm einst}}. It is interesting that the validity of the equation $J=0$ follows from $B^{\varphi}=0$ for the harmonic-Ricci solitons in consideration, as pointed out at the end of \hyperref[rmk bach flat gen rel]{Remark \ref*{rmk bach flat gen rel}}.
	
	The techniques used to prove our results are very close to one used by the authors of the articles we mentioned, with the difference that one have to take care to the geometry of the smooth map $\varphi$ during all the process. This is not particularly difficult once one understands the relevant role played by the $\varphi$-curvatures. Although we were able to generalize in the context of harmonic-Ricci solitons some of the results of \cite{PW}, \cite{FG}, \cite{MS} and \cite{CC} we are far from the more satisfactory description of rigidity of Ricci solitons. What is missing, up to now, is a more detailed study in the lower dimensional case $m=2,3$ (that will be object a future study) and the detection of a condition comparable notion to the one of locally conformally flatness for Ricci solitons.
	
	The paper is organized as follows. In \hyperref[section notation prel]{Section \ref*{section notation prel}} we fix the notations and we collect the necessary preliminaries on $\varphi$-curvatures, harmonic-Einstein manifolds and harmonic-Ricci solitons, providing the adequate references where one may find the proofs. In \hyperref[sect calc]{Section \ref*{sect calc}} we obtain some general formulas that will be used in the next Sections. \hyperref[section 4]{Section \ref*{section 4}} marks the beginning of the core of the paper: it is dedicated to the definition of rigidity for harmonic-Ricci solitons and the study of their property. The characterization of rigidity via the vanishing of the $\varphi$-Cotton tensor is the aim of \hyperref[section solitoni compatti]{Section \ref*{section solitoni compatti}} and \hyperref[section noncompac shirnking]{Section \ref*{section noncompac shirnking}}. In \hyperref[section solitoni compatti]{Section \ref*{section solitoni compatti}} we deal with not necessarily gradient compact solitons while in \hyperref[section noncompac shirnking]{Section \ref*{section noncompac shirnking}} we deal with complete non-compact shrinking solitons. In the final Section, \hyperref[section bach flat]{Section \ref*{section bach flat}}, we characterize rigidity via vanishing condition on the $\varphi$-Bach tensor.
	
	\section{Notations and preliminaries}\label{section notation prel}
	
	All the manifolds in this paper are assumed to be smooth, connected and without boundary.
	
	Let $M$ be a smooth manifold and $g$ be a Riemannian metric on $M$, we denote by $\nabla$ the Levi-Civita connection of $(M,g)$. For the Riemann tensor $\mbox{Riem}$ of $(M,g)$ we use the sign conventions
	\begin{equation*}
	R(X,Y)Z=\nabla_X(\nabla_YZ)-\nabla_Y(\nabla_XZ)-\nabla_{[X,Y]}Z \quad \mbox{ for every } X,Y,Z\in\mathfrak{X}(M),
	\end{equation*}
	where $\mathfrak{X}(M)$ is the $\mathcal{C}^{\infty}(M)$-module of vector fields on $M$ and $[\,,\,]$ denotes the Lie bracket, and
	\begin{equation*}
	\mbox{Riem}(W,Z,X,Y)=g(R(X,Y)Z,W) \quad \mbox{ for every } X,Y,Z,W\in\mathfrak{X}(M).
	\end{equation*}
	
	Let $(N,\langle\,,\,\rangle_N)$ be a Riemannian manifold of dimension $n$ and $\varphi:M\to N$ a smooth map.
	
	All the computation will be carry on using the moving frame formalism introduced by E. Cartan and we refer to Chapter 1 of \cite{AMR}. We fix the indexes ranges
	\begin{equation*}
	1\leq i,j,k,t,\ldots \leq m, \quad 1\leq a,b,c,d\ldots \leq n
	\end{equation*}
	and from now on we adopt the Einstein summation convention over repeated indexes. In a neighborhood of each point of $M$ we can write
	\begin{equation*}
	g=\delta_{ij}\theta^i\otimes \theta^j=\theta^i\otimes \theta^i
	\end{equation*}
	where $\delta_{ij}$ is the Kronecker delta and $\{\theta^i\}$ is a local orthonormal coframe. The dual frame will be denoted $\{e_i\}$ and is an orthonormal frame with respect to $g$.
	
	The Levi-Civita connection forms $\{\theta^i_j\}$ are characterized by the skew-symmetry $\theta^i_j+\theta^j_i=0$ and the validity of the first structure equations $d\theta^i+\theta^i_j\wedge \theta^j=0$ and the curvature forms $\{\Theta^i_j\}$ are given by $\Theta^i_j=\frac{1}{2}R^i_{jkt}\theta^k\wedge \theta^t$, where $R_{ijkt}$ are the components of the $(0,4)$-version of the Riemann tensor $\mbox{Riem}$ of $(M,g)$,
	\begin{equation*}
	\mbox{Riem}=R_{ijkt}\theta^i\otimes \theta^j\otimes \theta^k\otimes \theta^t,
	\end{equation*}
	and they satisfy the second structure equations $\Theta^j_i=d\theta^j_i+\theta^k_i\wedge\theta_k^j$.
	
	The Riemann tensor present the symmetries $R_{ijkt}+R_{ijtk}=0$, $R_{ijkt}+R_{jikt}=0$ and $R_{ijkt}=R_{ktij}$. Moreover it satisfy the two Bianchi identities
	\begin{equation*}
	R_{ijkt}+R_{iktj}+R_{itjk}=0
	\end{equation*}
	and
	\begin{equation}\label{second Bianchi identity}
	R_{ijkt,l}+R_{ijtl,k}+R_{ijlk,t}=0,
	\end{equation}
	where, for an arbitrary tensor field $T$ of type $(r,s)$
	\begin{equation*}
	T=T_{j_1\ldots j_s}^{i_1\ldots i_r}\theta^{j_1}\otimes \ldots \otimes \theta^{j_s}\otimes e_{i_1}\otimes \ldots \otimes e_{i_r},
	\end{equation*}
	its covariant derivative is defined as the tensor field of type $(r,s+1)$
	\begin{equation*}
	\nabla T=T_{j_1\ldots j_s,k}^{i_1\ldots i_r}\theta^k\otimes \theta^{j_1}\otimes \ldots \otimes \theta^{j_s}\otimes e_{i_1}\otimes \ldots \otimes e_{i_r},
	\end{equation*}
	where its components satisfies
	\begin{equation*}
	T_{j_1\ldots j_s,k}^{i_1\ldots i_r}\theta^k=dT_{j_1\ldots j_s}^{i_1\ldots i_r}-\sum_{t=1}^sT_{j_1\ldots j_{t-1} h j_{t+1}\ldots  j_s}^{i_1\ldots i_r}\theta^{h}_{j_t}+\sum_{t=1}^rT_{j_1\ldots j_s}^{i_1\ldots i_{t-1} h i_{t+1} \ldots i_r}\theta^{i_t}_h.
	\end{equation*}
	
	Later on we will make use of the following commutation relation 
	\begin{equation*}
	T_{j_1\ldots j_s,kt}^{i_1\ldots i_r}=T_{j_1\ldots j_s,tk}^{i_1\ldots i_r}+\sum_{t=1}^sR^h_{j_tkt}T_{j_1\ldots j_{t-1} hj_{t+1}\ldots j_s}^{i_1\ldots i_r}-\sum_{t=1}^rR^{i_t}_{hkt}T_{j_1\ldots j_s}^{i_1\ldots i_{t-1} h i_{t+1}\ldots i_r},
	\end{equation*}
	that when $T$ is given by the gradient of a smooth function $f$ yields
	\begin{equation}\label{comm rule der terza funzione}
	f_{ijk}=f_{ikj}+R^t_{ijk}f_t,
	\end{equation}
	and when $T$ is a two times covariant tensor field reads
	\begin{equation}\label{comm rule two times cov tensor field}
	T_{ij,kt}=T_{ij,tk}+R^l_{ikt}T_{lj}+R^l_{jkt}T_{il}.
	\end{equation}
	
	The the trace of the Riemann tensor is called Ricci tensor $\mbox{Ric}$ of $(M,g)$ and, in a local orthonormal coframe $\{\theta^i\}$, is given by
	\begin{equation}\label{ricci tensor trace riem}
	\mbox{Ric}=R_{ij}\theta^i\otimes \theta^j, \quad R_{ij}=R_{kikj}.
	\end{equation}
	The scalar curvature $S$ of $(M,g)$ is defined as the trace of the Ricci tensor and it is locally given by $S=\eta^{ij}R_{ij}$. Finally, the Riemannian volume element of $(M,g)$ is locally given by $\mu=\theta^1\land \ldots \land \theta^m$.
	
	Let $\{E_a\}$, $\{\omega^a\}$, $\{\omega^a_b\}$, $\{\Omega^a_b\}$ be an orthonormal frame, coframe, the respectively Levi-Civita connection forms and curvature forms on an open subset $\mathcal{V}$ on $N$ such that $\varphi^{-1}(\mathcal{V})\subseteq \mathcal{U}$. We set
	\begin{equation*}
	\varphi^*\omega^a=\varphi^a_i\theta^i
	\end{equation*}
	so that the differential $d\varphi$ of $\varphi$, a $1$-form on $M$ with values in the pullback bundle $\varphi^{-1}TN$, can be written as
	\begin{equation*}
	d\varphi=\varphi^a_i\theta^i\otimes E_a.
	\end{equation*}
	The generalized second fundamental tensor of the map $\varphi$ is given by $\nabla d\varphi$, locally
	\begin{equation*}
	\nabla d\varphi=\varphi^a_{ij}\theta^j\otimes \theta^i\otimes E_a,
	\end{equation*}
	where its coefficient are defined according to the rule
	\begin{equation*}
	\varphi^a_{ij}\theta^j=d\varphi^a_i-\varphi^a_k\theta^k_i+\varphi^b_i\omega^a_b.
	\end{equation*}
	The tension field $\tau(\varphi)$ of the map $\varphi$ is the section of $\varphi^{-1}TN$ given by
	\begin{equation}\label{def tension field}
	\tau(\varphi)=\mbox{tr}(\nabla d\varphi)=\varphi^a_{ii}E_a.
	\end{equation}
	Finally the bi-tension field $\tau_2(\varphi)$ of the map $\varphi$ is the section of $\varphi^{-1}TN$ with components
	\begin{equation}\label{def bi tension}
	\tau_2(\varphi)^a=\varphi^a_{iijj}-{}^NR^a_{bcd}\varphi^b_i\varphi^c_i\varphi^d_{jj},
	\end{equation}
	where ${}^NR^a_{bcd}$ are the components of the Riemann tensor of $(N,\langle\,,\,\rangle_N)$. Recall that $\varphi$ is said to harmonic if $\tau(\varphi)=0$ and bi-harmonic if $\tau_2(\varphi)=0$. Clearly harmonic maps are bi-harmonic.
	
	We denote by $\Delta$ the Laplace-Beltrami operator $\Delta u=\mbox{tr}(\mbox{Hess}(u))$ acting on functions $u:M\to \mathbb{R}$, where $\mbox{Hess}(u)=u_{ij}\theta^i\otimes \theta^j$, that is, $\Delta u=u_{ii}$. Moreover, for every $u\in\mathcal{C}^{\infty}(M)$ the $f$-Laplacian of $u$ is given by
	\begin{equation*}
	\Delta_f u:=e^f\mbox{div}(e^{-f}\nabla u).
	\end{equation*}
	It is easy to see that
	\begin{equation*}
	\Delta_f u=\Delta u-\langle \nabla f,\nabla u\rangle=u_{kk}-f_ku_k.
	\end{equation*}
	More generally, for every $X\in\mathfrak{X}(M)$ we set
	\begin{equation*}
		\Delta_X u=\Delta u-\langle X,\nabla u\rangle.
	\end{equation*}
	
	If $A$ is a symmetric two times covariant tensor on $(M,g)$ the totally traceless part of $A$ is given by
	\begin{equation*}
	\mathring{A}:=A-\frac{\mbox{tr}(A)}{m}g
	\end{equation*}
	and the symmetric two times covariant tensors $A^2$, $\Delta A$ and $\Delta_fA$ have components, in a local orthonormal coframe, respectively
	\begin{equation*}
	A^2_{ij}=A_{ik}A_{kj}, \quad \Delta A_{ij}=A_{ij,kk}
	\end{equation*}
	and
	\begin{equation*}
		\Delta_fA_{ij}=A_{ij,kk}-f_kA_{ij,k}.
	\end{equation*}

	\subsection{$\varphi$-Curvatures, harmonic-Einstein manifolds and harmonic-Ricci solitons}\label{section phi curv}
	
	We recall the definition of $\varphi$-curvatures that we shall need later on, for their proof and other details we refer to \cite{ACR} or Section 1.2 of \cite{A}.
	
	Let $(M,g)$ be a Riemannian manifold of dimension $m\geq 2$, $(N,\langle\,,\,\rangle_N)$ a target Riemannian manifold and $\alpha$ a positive real constant. The $\varphi$-Ricci tensor is defined as
	\begin{equation}\label{def phi Ricci}
	\mbox{Ric}^{\varphi}:=\mbox{Ric}-\alpha \varphi^*\langle\,,\,\rangle_N
	\end{equation}
	and its trace is denoted $S^{\varphi}$ and is called $\varphi$-scalar curvature. The $\varphi$-Schouten tensor is given by
	\begin{equation}\label{def of phi schouten}
	 A^{\varphi}:=\mbox{Ric}^{\varphi}-\frac{S^{\varphi}}{2(m-1)}g
	\end{equation}
	and the $\varphi$-Cotton tensor, that measures the failure of the commutation of the covariant derivatives of the $\varphi$-Schouten tensor, in global notation is given by
	\begin{equation*}
	C^{\varphi}(X,Y,Z):=\nabla_ZA^{\varphi}(X,Y)-\nabla_YA^{\varphi}(X,Z) \quad \mbox{ for every } X,Y,Z\in\mathfrak{X}(M)
	\end{equation*}
	while, in moving frame notation, its components in a local orthonormal coframe are given by
	\begin{equation}\label{def of phi Cotton}
	C^{\varphi}_{ijk}=A^{\varphi}_{ij,k}-A^{\varphi}_{ik,j},
	\end{equation}
	where we denoted by $A^{\varphi}_{ij}$ the components of the $\varphi$-Schouten tensor in a local orthonormal coframe.
	
	For manifolds of dimension $m\geq 3$ we are able to define the $\varphi$-Weyl tensor by
	\begin{equation}\label{def of phi Weyl}
	W^{\varphi}:=\mbox{Riem}-\frac{1}{m-2}A^{\varphi}\circleland g,
	\end{equation}
	where $\circleland$ denotes the Kulkarni-Nomizu product of two times covariant symmetric tensors, that is,
	\begin{equation*}
	(T\circleland V)(X,Y,Z,W)=T(X,Z)V(Y,W)-T(X,W)V(Y,Z)+T(Y,W)V(X,Z)-T(Y,Z)V(X,W).
	\end{equation*}
	In a local orthonormal coframe
	\begin{equation}\label{comp phi weyl}
		W^{\varphi}_{tijk}=R_{tijk}-\frac{1}{m-2}(A^{\varphi}_{tj}\delta_{ik}-A^{\varphi}_{tk}\delta_{ij}+A^{\varphi}_{ik}\delta_{tj}-A^{\varphi}_{ij}\delta_{tk}).
	\end{equation}

	We point out that the divergence of the $\varphi$-Weyl is not, in general, a multiple of $\varphi$-Cotton tensor. Indeed the following holds
	\begin{equation}\label{div di phi Weyl}
	W^{\varphi}_{tijk,t}=\frac{m-3}{m-2}C^{\varphi}_{ikj}+\alpha (\varphi^a_{ij}\varphi^a_k-\varphi^a_{ik}\varphi^a_j)+\frac{\alpha}{m-2}\varphi^a_{tt}(\varphi^a_j\delta_{ik}-\varphi^a_k\delta_{ij}).
	\end{equation}
	The traces of $\varphi$-Cotton and of $\varphi$-Weyl are given by, respectively, 
	\begin{equation}\label{traccia phi cotton}
	C^{\varphi}_{jji}=\alpha \varphi^a_{jj}\varphi^a_i
	\end{equation}
	and
	\begin{equation}\label{traccia phi weyl}
	W^{\varphi}_{kikj}=\alpha\varphi^a_i\varphi^a_j.
	\end{equation}

	Following P. Baird and J. Eells, see \cite{BaE}, we define the stress-energy tensor of $\varphi$ (with a different sign convention) by
	\begin{equation}\label{def tensore en stress}
	T:=\varphi^*\langle\,,\,\rangle_N-\frac{|d\varphi|^2}{2}g,
	\end{equation}
	where $|d\varphi|^2=\mbox{tr}(\varphi^*\langle\,,\,\rangle_N)$ is the square of the Hilbert-Schmidt norm of $d\varphi$. It is easy to see that, in a local orthonormal coframe,
	\begin{equation}\label{div tensore en stress}
	\mbox{div}(T)_j=\varphi^a_{ii}\varphi^a_j.
	\end{equation}
	A map $\varphi$ is called conservative if the energy-stress tensor $T$ is divergence free. From the formula above harmonic maps are conservative.
	
	The generalized Schur's identity is given by
	\begin{equation*}
	\mbox{div}(\mbox{Ric}^{\varphi})=\frac{1}{2}dS^{\varphi}-\alpha \mbox{div}(T),
	\end{equation*}
	locally
	\begin{equation}\label{div of phi Ricci}
		R^{\varphi}_{ij,j}=\frac{1}{2}S^{\varphi}_i-\alpha \varphi^a_{jj}\varphi^a_i,
	\end{equation}
	where $R^{\varphi}_{ij}$ are the components of the $\varphi$-Ricci tensor in a local orthonormal coframe.

	Finally, the $\varphi$-Bach tensor $ B^{\varphi}$ has components, in a local orthonormal coframe and for manifolds of dimension $m\geq 3$,
	\begin{equation}\label{def phi bach}
	(m-2)B^{\varphi}_{ij}=C^{\varphi}_{ijk,k}+R^{\varphi}_{tk}(W^{\varphi}_{tikj}-\alpha \varphi^a_t\varphi^a_i\delta_{jk})+\alpha\left(\varphi^a_{ij}\varphi_{kk}^a-\varphi^a_{kkj}\varphi^a_i-\frac{1}{m-2}|\tau(\varphi)|^2\delta_{ij}\right).
	\end{equation}
	It is not immediate to see but the $\varphi$-Bach tensor is symmetric and its trace is given by
	\begin{equation}\label{traccia phi bach}
	(m-2)\mbox{tr}(B^{\varphi})=\alpha\frac{m-4}{m-2}|\tau(\varphi)|^2.
	\end{equation}
	
	It remains only to define the tensor field $J$: its components are given by
	\begin{equation}\label{def di J}
	J^a:=\frac{mS^{\varphi}}{(m-1)(m-2)}\varphi_{ii}^a-\frac{m-2}{2(m-1)}S^{\varphi}_i\varphi^a_i-2R^{\varphi}_{ij}\varphi^a_{ij}+2\tau(\varphi)^b\varphi^b_i\varphi^a_i-\tau_2(\varphi)^a.
	\end{equation}
	For the motivation that led to its definition we refer to \cite{A20}.
	
	\begin{dfn}
		Let $(M,g)$ be Riemannian manifold of dimension $m\geq 2$, $\alpha$ a positive constant, $\varphi:M\to N$ a smooth map, where the target $(N,\langle\,,\,\rangle_N)$ is a Riemannian manifold, and $\lambda\in \mathbb{R}$. Then $(M,g)$ is called {\em harmonic-Einstein} (with respect to $\alpha$, $\varphi$ and $\lambda$) if
		\begin{equation}\label{harm eoinst non intro}
		\begin{dcases}
		\mbox{Ric}^{\varphi}=\lambda g\\
		\tau(\varphi)=0.
		\end{dcases}
		\end{equation}
		In case $\lambda=0$ we say that $(M,g)$ is {\em $\varphi$-Ricci flat} with respect to $\alpha$.
	\end{dfn}
	\begin{rmk}
		Notice that harmonic-Einstein manifolds have parallel $\varphi$-Ricci tensor and thus they are $\varphi$-Cotton flat. Moreover it is possible to see that they satisfy $B^{\varphi}=0$ and $J=0$.
	\end{rmk}
	
	\begin{dfn}
		Let $(M,g)$ be Riemannian manifold of dimension $m\geq 2$, $\alpha$ a positive constant, $\varphi:M\to N$ a smooth map, where the target $(N,\langle\,,\,\rangle_N)$ is a Riemannian manifold, $X\in\mathfrak{X}(M)$ and $\lambda\in \mathbb{R}$. Then $(M,g)$ is called {\em harmonic-Ricci soliton} (with respect to $\alpha$, $\varphi$, $X$ and $\lambda$) if 
		\begin{equation}\label{harm ricci soliton con campo vettoriale non intro}
		\begin{dcases}
		\mbox{Ric}^{\varphi}+\frac{1}{2}\mathcal{L}_Xg=\lambda g\\
		\tau(\varphi)=d\varphi(X).
		\end{dcases}
		\end{equation}
		The harmonic-Ricci soliton is called {\em shrinking}, {\em steady} or {\em expanding} if, respectively, $\lambda>0$, $\lambda=0$ or $\lambda<0$.
	\end{dfn}
	Every harmonic-Einstein endowed with a {\em vertical Killing vector field} $Y\in\mathfrak{X}(M)$, i.e., a solution of
	\begin{equation}\label{vertical killing vector field}
		\begin{dcases}
		\mathcal{L}_Yg=0\\
		d\varphi(Y)=0,
		\end{dcases}
	\end{equation}
	is trivially a harmonic-Ricci soliton. Furthermore, in \eqref{harm ricci soliton con campo vettoriale non intro}, the vector field $X$ can be replaced with $X+Y$, for a vertical Killing vector field $Y$. In case $X=\nabla f+Y$ for some vertical Killing vector field $Y\in\mathfrak{X}(M)$, the equations \eqref{harm ricci soliton con campo vettoriale non intro} can be written as
	\begin{equation}\label{grad harm ricci soliton non intro}
	\begin{dcases}
	\mbox{Ric}^{\varphi}+\mbox{Hess}(f)=\lambda g\\
	\tau(\varphi)=d\varphi(\nabla f)
	\end{dcases}
	\end{equation}
	and we say that $(M,g)$ is a {\em gradient harmonic-Ricci soliton}. The function $f$ is called {\em potential function} and it is defined up to an additive constant.
		
	\subsection{Preliminaries}
	
	We list the statement of some results that shall be useful in the rest of the article.
	
	We start with Theorem 5.1.1 of \cite{A}.
	\begin{thm}\label{thm with S phi constant is the manifold is a sphere with mu equal to zero}
		Let $(M,g)$ be a compact Riemannian manifold of dimension $m\geq 2$ with an Einstein-type structure of the form
		\begin{equation}\label{einstein type structure with mu equal to zero and X vector field in compact case}
		\begin{dcases}
		\mbox{Ric}^{\varphi}+\frac{1}{2}\mathcal{L}_Xg=\lambda g\\
		\tau(\varphi)=d\varphi(X),
		\end{dcases}
		\end{equation}
		for some $X\in\mathfrak{X}(M)$, $\lambda\in\mathcal{C}^{\infty}(M)$, $\alpha>0$ and $\varphi:M\to N$ a smooth map with target a Riemannian manifold $(N,\langle\,,\,\rangle_N)$. If $S^{\varphi}$ is constant then the structure \eqref{einstein type structure with mu equal to zero and X vector field in compact case} reduces to a harmonic-Einstein structure, that is,
		\begin{equation*}
		\begin{dcases}
		\mbox{Ric}^{\varphi}=\frac{S^{\varphi}}{m} g\\
		\tau(\varphi)=0
		\end{dcases}
		\end{equation*}
		with $S^{\varphi}$ constant.
	\end{thm}

	The fundamental step in the proof of the above Theorem is to show that integrating the equation
	\begin{equation}\label{X laplacian per strutt tipo Einstein}
	\frac{1}{2}\Delta_XS^{\varphi}=-\alpha|\tau(\varphi)|^2-|\mathring{\mbox{Ric}}^{\varphi}|^2-(S^{\varphi}-m\lambda)\frac{S^{\varphi}}{m}+(m-1)\Delta\lambda,
	\end{equation}
	one gets the validity of
	\begin{equation*}
	\frac{m-2}{2m}\int_M\langle X,\nabla S^{\varphi}\rangle=\int_M(|\mathring{\mbox{Ric}}^{\varphi}|^2+\alpha |\tau(\varphi)|^2).
	\end{equation*}
	
	\begin{comment}
		The following is Proposition 4.1.21 in \cite{A}.
		\begin{prop}\label{prop vertical killing vector field}
		Let $(M,\langle\,,\,\rangle)$ be a compact Riemannian manifold of dimension $m\geq 2$, $\varphi:M\to (N,\langle\,,\,\rangle_N)$ be a smooth map and $\alpha\in\mathbb{R}\setminus\{0\}$. Let $X$ be a vertical Killing vector field. If $\mbox{Ric}^{\varphi}\leq 0$ then $X$ is parallel. Further, if $\mbox{Ric}^{\varphi}$ is strictly negative at a point $x_0\in M$, then $X=0$.
		\end{prop}
	\end{comment}
	
	The following is part of Theorem 7.3.3 of \cite{A}. Its proof follows closely the one of Theorem 8.6 of \cite{AMR} and, once again, relies on the validity of \eqref{X laplacian per strutt tipo Einstein} and a clever use of the maximum principle.
	\begin{thm}\label{thm stime phi curv scalare nella tesi}
		Let $(M,g)$ be a complete  gradient harmonic-Ricci soliton of dimension $m\geq 2$ with respect to $\varphi:M\to N$ smooth map, where $(N,\langle\,,\,\rangle_N)$ is a Riemannian manifold, $f\in\mathcal{C}^{\infty}(M)$, $\alpha>0$ and $\lambda\geq 0$. Denoting $S^{\varphi}_*:=\inf_MS^{\varphi}$ we have $S^{\varphi}_*>-\infty$. Moreover
		\begin{itemize}
			\item[i)] If $\lambda=0$ then
			\begin{equation*}
			S^{\varphi}_*=0.
			\end{equation*}
			Then either $S^{\varphi}>0$ on $M$ or, if $f$ is non constant, $(M,g)$ splits as the Riemannian product of $\mathbb{R}$ with a totally geodesic $\psi$-Ricci flat hypersurface $\Sigma$, where $\psi:=\left.\varphi\right|_{\Sigma}$. Moreover $\varphi=\psi\circ \pi_{\Sigma}$ on $\mathbb{R}\times \Sigma$, where $\pi_{\Sigma}:\mathbb{R}\times \Sigma\to \Sigma$ is the canonical projection and the function $f$ can be expressed on $\mathbb{R}\times \Sigma$ as
			\begin{equation}
			f(t,x)=at+b \quad \mbox{ for every } t\in \mathbb{R} \mbox{ and } x\in\Sigma,
			\end{equation}
			for some $a>0$ and $b\in\mathbb{R}$ such that $\Sigma=f^{-1}(\{b\})$.
			\item[ii)] If $\lambda>0$ then
			\begin{equation}\label{stima inf curv scalare per lambda positivo}
			0\leq S^{\varphi}_*\leq m\lambda.
			\end{equation}
			
			If there exists $x_0\in M$ such that $S^{\varphi}(x_0)=0$ then $(M,g)$ is isometric to the euclidean space $\mathbb{R}^m$ and $\varphi$ is a constant map. Moreover, the potential $f$ can be expressed on $\mathbb{R}^m$ as $f(x)=\frac{\lambda}{2}|x|^2+\langle b,x\rangle+c$ for some $b\in\mathbb{R}^m$ and $c\in\mathbb{R}$, for every $x\in\mathbb{R}^m$.
			
			If $S^{\varphi}_*=m\lambda$ either $S^{\varphi}>m\lambda$ or $M$ is compact and $f$ is constant.
		\end{itemize}
	\end{thm}

	We combine Theorem 1.1 and Proposition 4.1 of \cite{YS} in a single statement.
	\begin{thm}
		Let $(M,g)$ be a complete non-compact gradient shrinking harmonic-Ricci soliton. Then for every point $p\in M$ there exist positive constants $C$ and $c$ independent from $R$ and $x$, respectively, such that, for $R$ sufficiently large
		\begin{equation*}
			\mbox{vol}(B_p(R))\leq CR^m,
		\end{equation*}
		and, for every $x\in M$,
		\begin{equation}\label{stima potenziale per solitone ricci harm}
			\frac{\lambda}{2}(r(x)-c)^2\leq f(x)\leq \frac{\lambda}{2}(r(x)+c)^2,
		\end{equation}
		where $B_p(R)$ is the geodesic ball of centre $p$ and radius $r$ and $r(x)$ is the geodesic distance from $x\in M$ to $p$.
	\end{thm}
	\begin{rmk}
		F. Yang and J. Shen in \cite{YS} deals with the normalized case where $\lambda=\frac{1}{2}$, but rescaling the metric clearly one can obtain the result for any $\lambda>0$. 
	\end{rmk}
	\begin{rmk}\label{rmk su properness potential funct}
		The estimate \eqref{stima potenziale per solitone ricci harm} shows that the potential function of a complete non-compact gradient shrinking harmonic-Ricci soliton is proper.
	\end{rmk}

	\begin{rmk}
		For a complete gradient shrinking harmonic-Ricci soliton of dimension $m\geq 2$ we have the following Hamilton-type identity
		\begin{equation}\label{ham type id}
		S^{\varphi}+|\nabla f|^2-2\lambda f \quad \mbox{ is constant on } M,
		\end{equation}
		see for instance equation $(7.1.7)$ of Proposition 7.1.5 in \cite{A}. Hence, by adding a suitable constant to the potential function, we may assume
		\begin{equation}\label{stima curv scalar per scelta opportuna potenziale}
		S^{\varphi}+|\nabla f|^2=2\lambda f.
		\end{equation}
		Then, for every $x\in M$,
		\begin{equation}\label{stima phi curv scalare come monteanu}
		0\leq S^{\varphi}(x)\leq \lambda^2(r(x)+c)^2,
		\end{equation}
		indeed, combining \eqref{stima inf curv scalare per lambda positivo}, \eqref{stima curv scalar per scelta opportuna potenziale} and \eqref{stima potenziale per solitone ricci harm}, one has the following chain of inequalities
		\begin{equation*}
		0\leq S^{\varphi}_*\leq S^{\varphi}=2\lambda f-|\nabla f|^2\leq 2\lambda f\leq \lambda^2(r+c)^2.
		\end{equation*}
		
		Now, since from \eqref{stima phi curv scalare come monteanu} and \eqref{stima potenziale per solitone ricci harm} both the potential function and the $\varphi$-scalar curvature have polynomial growth and the volume of $M$ is at most Euclidean, it is clear that for every $\mu,\gamma\in\mathbb{R}$
		\begin{equation}\label{finitezza inte di s phi e nabla f quadro}
		\int_M(S^{\varphi})^{\gamma}e^{-\mu f}<+\infty, \quad \int_M|\nabla f|^{\gamma}e^{-\mu f}<+\infty.
		\end{equation}
		The validity of \eqref{finitezza inte di s phi e nabla f quadro} will be crucial in \hyperref[section noncompac shirnking]{Section \ref*{section noncompac shirnking}}.
	\end{rmk}
	
	\section{Fundamental calculations}\label{sect calc}
	
	We denote by $F^{\varphi}$ the three-times covariant tensor representing the obstruction to the commutation of the covariant derivative of $\mbox{Ric}^{\varphi}$, more precisely we give
	\begin{dfn}
		Let $(M,g)$ be a Riemannian manifold of dimension $m\geq 2$, $\alpha$ a positive constant and $\varphi:M\to N$ a smooth map, where $(N,\langle\,,\,\rangle_N)$ is a target Riemannian manifold. The components of the three times covariant tensor field $F^{\varphi}$ are given, in a local orthonormal coframe, by
		\begin{equation}\label{def di F phi}
		F^{\varphi}_{ijk}:=R^{\varphi}_{ij,k}-R^{\varphi}_{ik,j}.
		\end{equation}
	\end{dfn}
	Recalling the definitions of the $\varphi$-Schouten tensor \eqref{def of phi schouten} and of the $\varphi$-Cotton tensor \eqref{def of phi Cotton} the following Proposition is easy to prove.
	\begin{prop}
		Let $(M,g)$ be a Riemannian manifold of dimension $m\geq 2$, $\alpha$ a positive constant and $\varphi:M\to N$ a smooth map, where $(N,\langle\,,\,\rangle_N)$ is a target Riemannian manifold. Then
		\begin{equation}\label{norma di F quadro relazionata a norma di C phi}
		|F^{\varphi}|^2=|C^{\varphi}|^2+\frac{2\alpha}{m-1}\mbox{div}(T)(\nabla S^{\varphi})+\frac{1}{2(m-1)}|\nabla S^{\varphi}|^2,
		\end{equation}
		where $T$ is the energy-stress tensor of the smooth map $\varphi$.
	\end{prop}
	\begin{proof}
		By the definition \eqref{def di F phi} of $F^{\varphi}$
		\begin{equation}\label{norma di F phi}
		|F^{\varphi}|^2=2F^{\varphi}_{ijk}R^{\varphi}_{ij,k}.
		\end{equation}
		Using \eqref{def of phi Cotton} and \eqref{def of phi schouten} we easily get
		\begin{equation*}
		F^{\varphi}_{ijk}=C^{\varphi}_{ijk}+\frac{1}{2(m-1)}(S^{\varphi}_k\delta_{ij}-S^{\varphi}_j\delta_{ik}),
		\end{equation*}
		and by plugging it into the above we get
		\begin{equation*}
		|F^{\varphi}|^2=2R^{\varphi}_{ij,k}C^{\varphi}_{ijk}+\frac{1}{m-1}(S^{\varphi}_kR^{\varphi}_{ii,k}-S^{\varphi}_jR^{\varphi}_{ij,i}).
		\end{equation*}
		Using once again the definition of $\varphi$-Schouten \eqref{def of phi schouten} combined with \eqref{div of phi Ricci} from the above we infer
		\begin{equation*}
		|F^{\varphi}|^2=2A^{\varphi}_{ij,k}C^{\varphi}_{ijk}+\frac{1}{m-1}S^{\varphi}_kC^{\varphi}_{iik}+\frac{1}{m-1}|\nabla S^{\varphi}|^2-\frac{1}{m-1}S^{\varphi}_j\left(\frac{1}{2}S^{\varphi}_j-\alpha\varphi^a_{ii}\varphi^a_j\right).
		\end{equation*}
		Now, using \eqref{traccia phi cotton} and
		\begin{equation*}
		|C^{\varphi}|^2=2C^{\varphi}_{ij,k}A^{\varphi}_{ij,k},
		\end{equation*}
		from the above we easily get
		\begin{equation*}
		|F^{\varphi}|^2=|C^{\varphi}|^2+\frac{2\alpha}{m-1}S^{\varphi}_j\varphi^a_{ii}\varphi^a_j+\frac{1}{2(m-1)}|\nabla S^{\varphi}|^2.
		\end{equation*}
		Then \eqref{norma di F quadro relazionata a norma di C phi} follows, recalling the validity of \eqref{div tensore en stress}.
	\end{proof}
	
	\begin{rmk}
		From \eqref{norma di F quadro relazionata a norma di C phi} we deduce
		\begin{itemize}
			\item[i)] If $A^{\varphi}$ is a Codazzi tensor, i.e., $C^{\varphi}=0$, we have that $\varphi$ is conservative (since the trace of $\varphi$-Cotton is a constant multiple of the divergence of $T$) and thus from \eqref{norma di F quadro relazionata a norma di C phi}
			\begin{equation}\label{norma F phi quadro con C phi zero}
			|F^{\varphi}|^2=\frac{1}{2(m-1)}|\nabla S^{\varphi}|^2.
			\end{equation}
			\item[ii)] If $\mbox{Ric}^{\varphi}$ is a Codazzi tensor, i.e., if $F^{\varphi}=0$, we have from \eqref{norma di F quadro relazionata a norma di C phi} that $S^{\varphi}$ is constant and $A^{\varphi}$ is Codazzi.
		\end{itemize}
	\end{rmk}
	
	\begin{rmk}
		Assume that $\varphi$ is a constant map. We have that
		\begin{equation}\label{F phi in corrispondenza a riem phi cost}
		F^{\varphi}=-\mbox{div}(\mbox{Riem}),
		\end{equation}
		indeed the second Bianchi identity \eqref{second Bianchi identity} implies
		\begin{equation*}
		R_{tikj,t}=R_{ij,k}-R_{ik,j}.
		\end{equation*}
		
		Moreover, it is well known that
		\begin{equation*}
		\mbox{div}(W)=-\frac{m-3}{m-2}C,
		\end{equation*}
		see for instance $(1.87)$ of \cite{AMR}, hence when $m\geq 4$
		\begin{equation}\label{C phi in corrispondenza a weyl phi cost}
		C^{\varphi}=-\frac{m-2}{m-3}\mbox{div}(W).
		\end{equation}
		
		Relations \eqref{F phi in corrispondenza a riem phi cost} and \eqref{C phi in corrispondenza a weyl phi cost} shows how $F^{\varphi}=0$ and $C^{\varphi}=0$, in case $\varphi:M\to N$ is a non-constant smooth map, generalize the notions of harmonic curvature and harmonic Weyl tensor, respectively.
	\end{rmk}

	\begin{dfn}
		Let $(M,g)$ be a Riemannian manifold of dimension $m\geq 2$, $\alpha$ a positive constant and $\varphi:M\to N$ a smooth map, where $(N,\langle\,,\,\rangle_N)$ is a target Riemannian manifold. We say that $(M,g)$ has {\em parallel $\varphi$-Ricci tensor} if $\nabla \mbox{Ric}^{\varphi}=0$.
	\end{dfn}

	\begin{rmk}\label{rmk phi ricci parallelo}
		Assume $(M,g)$ has parallel $\varphi$-Ricci tensor, for some smooth map $\varphi:M\to N$, where $(N,\langle\,,\,\rangle_N)$ is a Riemannian manifold, and some real constant $\alpha\neq 0$. Then the $\varphi$-scalar curvature is constant and $\varphi$ is conservative. Indeed, since the $\varphi$-scalar curvature is the trace of the $\varphi$-Riemann tensor, one has
		\begin{equation*}
		S^{\varphi}_i=R^{\varphi}_{jj,i}=0,
		\end{equation*}
		and since $M$ is connected then $S^{\varphi}$ is constant on $M$. Then, using \eqref{div of phi Ricci} and the constancy of the $\varphi$-scalar curvature,
		\begin{equation*}
		0=R^{\varphi}_{ij,j}=\frac{1}{2}S^{\varphi}_i-\alpha\varphi^a_{jj}\varphi^a_i=-\alpha\varphi^a_{jj}\varphi^a_i,
		\end{equation*}
		and thus $\varphi$ is conservative.
	\end{rmk}
	
	\begin{rmk}\label{rmk soliton con phi conser in realta e arm}
		Let $\varphi:(M,g)\to (N,\langle\,,\,\rangle_N)$ be a conservative map between two Riemannian manifolds. Assume
		\begin{equation*}
		\tau(\varphi)=d\varphi(X)
		\end{equation*}
		for some $X\in\mathfrak{X}(M)$. Then, using \eqref{div tensore en stress}, we get
		\begin{equation*}
		|\tau(\varphi)|^2=\mbox{div}(T)(X)=0,
		\end{equation*}
		that is, $\varphi$ is harmonic.
	\end{rmk}
	\begin{rmk}\label{rmk soliton con phi ricci parallelo}
		Assume that a harmonic-Ricci soliton of dimension $m\geq 3$ has parallel $\varphi$-Ricci tensor. Then, as seen in \hyperref[rmk phi ricci parallelo]{Remark \ref*{rmk phi ricci parallelo}}, the $\varphi$-scalar curvature is constant and $\varphi$ is conservative and thus, from \hyperref[rmk soliton con phi conser in realta e arm]{Remark \ref*{rmk soliton con phi conser in realta e arm}} we have that $\varphi$ is harmonic. Then the components of $\varphi$-Bach are given by
		\begin{equation}\label{phi bach con phi ricci parallelo}
		(m-2)B^{\varphi}_{ij}=W^{\varphi}_{tikj}R^{\varphi}_{tk}-\alpha R^{\varphi}_{kj}\varphi^a_k\varphi^a_i.
		\end{equation}
		Indeed, $S^{\varphi}$ is constant $\nabla A^{\varphi}=\nabla \mbox{Ric}^{\varphi}=0$ and thus $C^{\varphi}=0$. Using also that $\varphi$ is harmonic from \eqref{def phi bach} we immediately get \eqref{phi bach con phi ricci parallelo}. Furthermore, since $S^{\varphi}$ is constant and $\varphi$ is harmonic (recall that a harmonic map is also bi-harmonic) in \eqref{def di J}, we also get
		\begin{equation}\label{J con phi ricci parallelo}
		J^a=-2R^{\varphi}_{jk}\varphi^a_{jk}.
		\end{equation}
	\end{rmk}	
	
	In the next Proposition we collect a list of useful formulas for gradient harmonic-Ricci solitons. Some of them are not new (see for instance \cite{A}), but we provide their proof for the reader convenience. This is not the case for \eqref{f laplacian phi ricci per solitone grad} (or equivalently, \eqref{f laplacian phi ricci per solitone grad scritt con hess f}), the most difficult to obtain, in terms of computation.
	\begin{prop}
		Let $(M,g)$ be a gradient harmonic-Ricci soliton with respect to $\varphi:M\to N$ smooth, where $(N,\langle\,,\,\rangle_N)$ is a target Riemannian manifold, $\alpha>0$ and $\lambda\in \mathbb{R}$. Then, the following formulas holds.
		\begin{equation}\label{form commutazione der covariante per ricci phi in solitoni gradiente}
		F^{\varphi}_{ijk}=R^{\varphi}_{ij,k}-R^{\varphi}_{ik,j}=R_{tikj}f_t;
		\end{equation}
		\begin{equation}\label{form gradiente phi curv scalare in solitoni gradiente}
		\frac{1}{2}S^{\varphi}_i=R^{\varphi}_{ij}f_j;
		\end{equation}
		\begin{equation}\label{div riem per solioni harm gradi}
		\mbox{div}(\mbox{Riem})_{ikj}=R_{tikj,t}=R_{tikj}f_t+\alpha(\varphi_{ik}\varphi^a_j-\varphi^a_{ij}\varphi^a_k);
		\end{equation}
		\begin{equation}\label{div riem con esponeziale caso compatto}
		(R_{tikj}e^{-f})_t=(\mbox{div}(\mbox{Riem})_{ikj}-f_tR_{tikj})e^{-f}=\alpha(\varphi_{ik}\varphi^a_j-\varphi^a_{ij}\varphi^a_k)e^{-f};
		\end{equation}
		\begin{equation}\label{norma F quadro caso compatto}
		f_tR_{tikj}R^{\varphi}_{ij,k}=\frac{1}{2}|F^{\varphi}|^2;
		\end{equation} 
		\begin{equation}\label{f laplacian phi ricci per solitone grad}
		\frac{1}{2}\Delta_fR^{\varphi}_{ij}=\lambda R^{\varphi}_{ij}+R_{tijk}R^{\varphi}_{tk}+\frac{\alpha}{2}\varphi^a_k(R^{\varphi}_{kj}\varphi^a_i+\varphi^a_jR^{\varphi}_{ik})- \alpha\varphi^a_{kk}\varphi^a_{ij};
		\end{equation}
		\begin{equation}\label{f laplacian phi ricci per solitone grad scritt con hess f}
		\frac{1}{2}\Delta_fR^{\varphi}_{ij}=-R_{tijk}f_{tk}-\frac{\alpha}{2}\varphi^a_k(f_{kj}\varphi^a_i+\varphi^a_jf_{ik})- \alpha\varphi^a_{kk}\varphi^a_{ij};
		\end{equation}
		\begin{equation}\label{f laplacian phi scalar curvature}
		\frac{1}{2}\Delta_fS^{\varphi}=\lambda S^{\varphi}-|\mbox{Ric}^{\varphi}|^2- \alpha|\tau(\varphi)|^2.
		\end{equation}
	\end{prop}
	\begin{proof}
		In a local orthonormal coframe the harmonic-Ricci soliton equations are given by
		\begin{equation}\label{gradient harm ricci soliton in a local orth coframe}
		\begin{dcases}
		R^{\varphi}_{ij}+f_{ij}=\lambda \delta_{ij}\\
		\varphi^a_{ii}=\varphi^a_if_i.
		\end{dcases}
		\end{equation}
		
		Taking the covariant derivative of the first equation of \eqref{gradient harm ricci soliton in a local orth coframe} we get
		\begin{equation*}
		R^{\varphi}_{ij,k}+f_{ijk}=0.
		\end{equation*}
		Skew-symmetrizing the above with respect to the indexes $j$ and $k$, using the commutation rule \eqref{comm rule der terza funzione} and recalling the definition \eqref{def di F phi} of $F^{\varphi}$ we infer the validity of \eqref{form commutazione der covariante per ricci phi in solitoni gradiente}.
		
		Summing on $i=j$ the  relation \eqref{form commutazione der covariante per ricci phi in solitoni gradiente} we obtain
		\begin{equation*}
		S^{\varphi}_k-\frac{1}{2}S^{\varphi}_k+\alpha \varphi^a_{ii}\varphi^a_k=R_{tk}f_t,
		\end{equation*}
		that is \eqref{form gradiente phi curv scalare in solitoni gradiente}, using the second equation of \eqref{gradient harm ricci soliton in a local orth coframe} and recalling that, by definition, $R^{\varphi}_{ij}=R_{ij}-\alpha \varphi^a_i\varphi^a_j$.
		
		The second Bianchi identity \eqref{second Bianchi identity} with $l=t$ gives
		\begin{equation*}
		0=R_{tikj,t}+R_{tijt,k}+R_{titk,j}=\mbox{div}(\mbox{Riem})_{ikj}-R_{ij,k}+R_{ik,j},
		\end{equation*}
		hence
		\begin{equation*}
		\mbox{div}(\mbox{Riem})_{ikj}=R_{ij,k}-R_{ik,j}.
		\end{equation*}
		Using the definition of $\varphi$-Ricci the above gives
		\begin{equation*}
		\mbox{div}(\mbox{Riem})_{ikj}=R^{\varphi}_{ij,k}-R^{\varphi}_{ik,j}+\alpha(\varphi_{ik}\varphi^a_j-\varphi^a_{ij}\varphi^a_k),
		\end{equation*}
		and thus \eqref{div riem per solioni harm gradi} follows with the aid of \eqref{form commutazione der covariante per ricci phi in solitoni gradiente}.
		
		From \eqref{div riem per solioni harm gradi} it is immediate to obtain \eqref{div riem con esponeziale caso compatto}.
		
		Using \eqref{form commutazione der covariante per ricci phi in solitoni gradiente} we get $R_{tikj}f_t=F^{\varphi}_{ijk}$, hence recalling the validity of \eqref{norma di F phi} we have
		\begin{equation*}
		f_tR_{tikj}R^{\varphi}_{ij,k}=F^{\varphi}_{ijk}R^{\varphi}_{ij,k}=\frac{1}{2}|F^{\varphi}|^2,
		\end{equation*} 
		that is, \eqref{norma F quadro caso compatto}.

		Now we prove the validity of \eqref{f laplacian phi ricci per solitone grad}. First of all, the following hold
		\begin{equation}\label{f gradiente di phi ricci in proof}
		R^{\varphi}_{ij,k}f_k=R_{tikj}f_tf_k+\frac{1}{2}S^{\varphi}_{ij}+R^{\varphi}_{ik}R^{\varphi}_{kj}-\lambda R^{\varphi}_{ij}.
		\end{equation}
		Indeed
		\begin{equation*}
			R^{\varphi}_{ij,k}f_k=(R^{\varphi}_{ij,k}-R^{\varphi}_{ik,j})f_k+R^{\varphi}_{ik,j}f_k=(R^{\varphi}_{ij,k}-R^{\varphi}_{ik,j})f_k+(R^{\varphi}_{ik}f_k)_j-R^{\varphi}_{ik}f_{kj},
		\end{equation*}
		and then \eqref{f gradiente di phi ricci in proof} follows immediately using \eqref{form commutazione der covariante per ricci phi in solitoni gradiente}, \eqref{form gradiente phi curv scalare in solitoni gradiente} and the first equation of \eqref{gradient harm ricci soliton in a local orth coframe} as follows
		\begin{align*}
		R^{\varphi}_{ij,k}f_k=&(R^{\varphi}_{ij,k}-R^{\varphi}_{ik,j})f_k+(R^{\varphi}_{ik}f_k)_j-R^{\varphi}_{ik}f_{kj}\\
		=&R_{tikj}f_tf_k+\left(\frac{1}{2}S^{\varphi}_i\right)_j-R^{\varphi}_{ik}(-R^{\varphi}_{kj}+\lambda \delta_{kj})\\
		=&R_{tikj}f_tf_k+\frac{1}{2}S^{\varphi}_{ij}+R^{\varphi}_{ik}R^{\varphi}_{kj}-\lambda R^{\varphi}_{ij}.
		\end{align*}
		Next, we claim the validity of
		\begin{equation}\label{laplaciano di phi ricci in proof}
		R^{\varphi}_{ij,kk}=R_{kjti}f_kf_t-2R_{tikj}R^{\varphi}_{tk}+\lambda R^{\varphi}_{ij}+\frac{1}{2}S^{\varphi}_{ij}+R^{\varphi}_{kj}R^{\varphi}_{ik}+\alpha\varphi^a_k(\varphi^a_jR^{\varphi}_{ik}+\varphi^a_iR^{\varphi}_{kj})-2\alpha \varphi^a_{kk}\varphi^a_{ij}.
		\end{equation}
		To prove the claim notice that, using \eqref{form commutazione der covariante per ricci phi in solitoni gradiente}, the commutation rule \eqref{comm rule two times cov tensor field} and Schur's lemma \eqref{div of phi Ricci},	
		\begin{align*}
			R^{\varphi}_{ij,kk}=&(R^{\varphi}_{ij,k}-R^{\varphi}_{ik,j})_k+R^{\varphi}_{ik,jk}\\
			=&(R_{tikj}f_t)_k+R^{\varphi}_{ik,kj}+R^t_{ijk}R^{\varphi}_{tk}+R^t_{kjk}R^{\varphi}_{it}\\
			=&R_{tikj,k}f_t+R_{tikj}f_{tk}+\left(\frac{1}{2}S^{\varphi}_i-\alpha \varphi^a_{kk}\varphi^a_i\right)_j+R_{tijk}R^{\varphi}_{tk}+R_{tj}R^{\varphi}_{it}.
		\end{align*}
		Then, using the definition of $\varphi$-Ricci and the symmetries of the Riemann tensor from the above we get
		\begin{align*}
		R^{\varphi}_{ij,kk}=&\mbox{div}(\mbox{Riem})_{jki}f_k+R_{tikj}f_{tk}+\frac{1}{2}S^{\varphi}_{ij}-\alpha (\varphi^a_{kkj}\varphi^a_i+\varphi^a_{kk}\varphi^a_{ij})+R_{tijk}R^{\varphi}_{tk}+R^{\varphi}_{kj}R^{\varphi}_{ik}+\alpha R^{\varphi}_{ik}\varphi^a_i\varphi^a_j .
		\end{align*}
		Now we plug \eqref{div riem per solioni harm gradi} into the above obtaining
		\begin{equation}\label{prima scritt per laplaciano phi ricci}
		\begin{aligned}
		R^{\varphi}_{ij,kk}=&R_{tjki}f_tf_k+\alpha(\varphi_{jk}f_k\varphi^a_i-\varphi^a_{ij}\varphi^a_kf_k)+R_{tikj}f_{tk}+\frac{1}{2}S^{\varphi}_{ij}\\
		&+\alpha (R^{\varphi}_{ik}\varphi^a_i\varphi^a_j-\varphi^a_{kkj}\varphi^a_i-\varphi^a_{kk}\varphi^a_{ij})+R_{tijk}R^{\varphi}_{tk}+R^{\varphi}_{kj}R^{\varphi}_{ik}.
		\end{aligned}
		\end{equation}
		Notice that the from the second equation of \eqref{gradient harm ricci soliton in a local orth coframe} we have
		\begin{equation*}
			\varphi^a_{kkj}=(\varphi^a_kf_k)_j=\varphi^a_{kj}f_k+\varphi^a_kf_{kj}.
		\end{equation*}
		Using the above and both the equations of \eqref{gradient harm ricci soliton in a local orth coframe} from \eqref{prima scritt per laplaciano phi ricci}, after a few simplifications, we infer
		\begin{align*}
		R^{\varphi}_{ij,kk}=&R_{tjki}f_tf_k-2\alpha\varphi^a_{ij}\varphi^a_{kk}-R_{tikj}R^{\varphi}_{tk}+\lambda R_{ij}+\frac{1}{2}S^{\varphi}_{ij}\\
		&+\alpha (R^{\varphi}_{ik}\varphi^a_i\varphi^a_j-\varphi^a_kf_{kj}\varphi^a_i)+R_{tijk}R^{\varphi}_{tk}+R^{\varphi}_{kj}R^{\varphi}_{ik},
		\end{align*}
		so that, using once again the symmetries of the Riemann tensor, the first equation of \eqref{gradient harm ricci soliton in a local orth coframe} and the definition of $\varphi$-Ricci we obtain the claimed equality \eqref{laplaciano di phi ricci in proof}.
		
		The validity of \eqref{f laplacian phi ricci per solitone grad} follows immediately from the definition $\Delta_fR^{\varphi}_{ij}=R^{\varphi}_{ij,kk}-R^{\varphi}_{ij,k}f_k$, using \eqref{laplaciano di phi ricci in proof} and \eqref{f gradiente di phi ricci in proof}.
		
		To get \eqref{f laplacian phi ricci per solitone grad scritt con hess f} it is sufficient to use the first equation of \eqref{gradient harm ricci soliton in a local orth coframe} and the definition of $\varphi$-Ricci into \eqref{f laplacian phi ricci per solitone grad}.
		
		Tracing \eqref{f laplacian phi ricci per solitone grad} we get
		\begin{equation*}
			\frac{1}{2}\Delta_fS^{\varphi}=\lambda S^{\varphi}-R_{tk}R^{\varphi}_{tk}+\alpha\varphi^a_kR^{\varphi}_{ki}\varphi^a_i- \alpha|\tau(\varphi)|^2,
		\end{equation*}
		that is, using the definition of the $\varphi$-Ricci tensor, \eqref{f laplacian phi scalar curvature}.
	\end{proof}

	\section{Rigidity}\label{section 4}
	
	Following what P. Petersen and W. Wylie did in \cite{PW}, we give the following
	\begin{dfn}\label{def rigidita}
		Let $(M,g)$ be a gradient harmonic-Ricci soliton with respect to a positive constant $\alpha$, a smooth map $\varphi:M\to N$, where $(N,\langle\,,\,\rangle_N)$ is a Riemannian manifold, $f\in\mathcal{C}^{\infty}(M)$ and $\lambda\in\mathbb{R}$. We say that $(M,g)$ is {\em rigid} if it is isometric to a quotient of the Riemannian product $L\times \mathbb{R}^k$, where
		\begin{itemize}
			\item[$i)$] $0\leq k\leq m$ is an integer and $\mathbb{R}^k$ is endowed with the canonical metric $g_{\mbox{can}}$,
			\item[$ii)$] $(L,g_L)$ is a harmonic-Einstein manifold with respect to $\alpha$, a smooth map $\varphi_L:L\to N$ and $\lambda\in\mathbb{R}$;
			\item[$iii)$] Via the isometry $\varphi$ is given the lifting of $\varphi_L$ and $f$ by the lifting of the map $f_{\mathbb{R}^k}(x)=\frac{\lambda}{2}|x|^2+\langle b,x\rangle +c$, for some $c\in\mathbb{R}$ and $b\in\mathbb{R}^k$, i.e., via the isometry $\varphi=\varphi_L\circ \pi_L$ and $f=f_{\mathbb{R}^k}\circ \pi_{\mathbb{R}^k}$, where $\pi_L:L\times \mathbb{R}^k\to L$ and $\pi_{\mathbb{R}^k}:L\times \mathbb{R}^k\to \mathbb{R}^k$ are the canonical projections.
		\end{itemize}
	\end{dfn}
	\begin{rmk}
		In the above:
		\begin{itemize}
			\item[i)] When $k=0$ what we get is that $(M,g)$ is isometric to a quotient of the harmonic-Einstein manifold $(L,g_L)$ and, via the isometry, $\varphi=\varphi_L\circ \pi_L$ and the potential is constant;
			\item[ii)] When $k=m$ what we get is that $(M,g)$ is isometric to a quotient of $(\mathbb{R}^m,g_{\mbox{can}})$ and, via the isometry, $f=f_{\mathbb{R}^k}\circ \pi_{\mathbb{R}^k}$ and $\varphi$ is constant.
		\end{itemize}
	\end{rmk}

	Clearly a compact gradient harmonic-Ricci soliton is rigid if and only if it is harmonic-Einstein. Motivated by this we extend the notion of rigidity to any (not necessarily gradient) harmonic-Ricci soliton, in the following
	\begin{dfn}\label{def rig compatto}
		Let $(M,g)$ be a compact harmonic-Ricci soliton with respect to a positive constant $\alpha$, a smooth map $\varphi:M\to N$, where $(N,\langle\,,\,\rangle_N)$ is a Riemannian manifold, $X\in\mathfrak{X}(M)$ and $\lambda\in\mathbb{R}$. We say that $(M,g)$ is {\em rigid} if it is harmonic-Einstein with respect to $\alpha$, $\varphi$ and $\lambda$.
	\end{dfn}
	
	We will deal to the characterization of rigidity of compact harmonic-Ricci soliton in \hyperref[section solitoni compatti]{Section \ref*{section solitoni compatti}} and thus, for the rest of the Section, we will focus on complete non-compact gradient solitons.

	The aim of the following Proposition is to justify \hyperref[def rigidita]{Definition \ref*{def rigidita}}, showing that the Riemannian products $L\times \mathbb{R}^k$ described in \hyperref[def rigidita]{Definition \ref*{def rigidita}} (for $1\leq k\leq m-1$) are actually harmonic-Ricci soliton. We will also study some of their geometric properties.  
	\begin{prop}\label{prop che descrive modello rigido}
		Let $(M,g)$ a harmonic-Einstein manifold of dimension $m\geq 2$ with respect to a positive constant $\alpha$, a smooth map $\varphi:M\to N$, where $(N,\langle\,,\,\rangle_N)$ is a target Riemannian manifold, and $\lambda\in\mathbb{R}$. Let $k\geq 1$ be an integer and consider on the Euclidean space $\mathbb{R}^k$, endowed with its canonical metric $g_{\mbox{can}}$, the function
		\begin{equation*}
		f(x):=\frac{\lambda}{2}|x|^2+\langle b,x\rangle+c \quad \mbox{ for every } x\in\mathbb{R}^k,
		\end{equation*}
		where $b\in\mathbb{R}^k$ and $c\in\mathbb{R}$. Consider the Riemannian product $\bar{M}:=M\times \mathbb{R}^k$, with Riemannian metric
		\begin{equation*}
		\bar{g}=\pi_M^*g+\pi_{\mathbb{R}^k}^*g_{\mbox{can}}\equiv g+g_{\mbox{can}},
		\end{equation*}
		where $\pi_M:\bar{M}\to M$ and $\pi_{\mathbb{R}^k}:\bar{M}\to \mathbb{R}^k$ are the canonical projections. Denote by $\bar{f}$ and $\bar{\varphi}$ the lifting of $f$ and $\varphi$, respectively, to $\bar{M}$:
		\begin{equation}\label{def di f bar e phi bar}
		\bar{f}:=f\circ \pi_{\mathbb{R}^k}\in\mathcal{C}^{\infty}(\bar{M}), \quad \bar{\varphi}:=\varphi\circ \pi_M:\bar{M}\to N.
		\end{equation}
		Then $(\bar{M},\bar{g})$ is a gradient harmonic-Ricci soliton of dimension $\bar{m}=m+k\geq 3$ with respect to $\alpha$, $\bar{\varphi}$, $\bar{f}$ and $\lambda$, that is,
		\begin{equation}\label{soliton equation for rigid model}
		\begin{cases}
		\overline{\mbox{Ric}}-\alpha \bar{\varphi}^*\langle\,,\,\rangle_N+\overline{\mbox{Hess}}(\bar{f})=\lambda \overline{g}\\
		\bar{\tau}(\bar{\varphi})=\bar{d}\bar{\varphi}(\bar{\nabla} \bar{f})
		\end{cases}
		\end{equation}
		holds. Moreover
		\begin{equation}\label{nabla phi ricci zero per prodotto}
			\bar{\nabla} \overline{\mbox{Ric}}^{\bar{\varphi}}=0,
		\end{equation}
		and, as a consequence, the $\bar{\varphi}$-scalar curvature is constant and $\bar{\varphi}$ is harmonic. Furthermore, 
		\begin{equation}\label{bar phi bach per modello rigid}
		\bar{B}^{\bar{\varphi}}=\frac{(k-1)\lambda^2}{(m+k-1)(m+k-2)^2}(k\pi^*_Mg-m\pi_{\mathbb{R}^k}^*g_{\mbox{can}})
		\end{equation}
		and
		\begin{equation}\label{bar J per modello rigid}
			\bar{J}=0.
		\end{equation}
	\end{prop}
	\begin{proof}
		
		We will use the following indexes conventions
		\begin{equation*}
		1\leq i,j,\ldots \leq m, \quad 1\leq \alpha,\beta,\ldots \leq k, \quad 1\leq A,B,\ldots \leq m+k.
		\end{equation*}
		Let $\{\theta^i\}$ be a local orthonormal coframe for $(M,g)$ on an open subset $\mathcal{U}$ of $M$ and  $\{\psi^{\alpha}\}$ for $(\mathbb{R}^k,g_{\mbox{can}})$ on an open subset $\mathcal{W}$ of $\mathbb{R}^k$. It is easy to see that, by setting
		\begin{equation*}
			\bar{\theta}^i:=\pi_M^*\theta^i, \quad \bar{\theta}^{m+\alpha}:=\pi^*_{\mathbb{R}^k}\psi^{\alpha},
		\end{equation*}
		then $\{\bar{\theta}^A\}$ is a local orthonormal coframe for $\bar{g}$ in $\overline{\mathcal{U}}:=\mathcal{U}\times \mathcal{W}$.  The same applies for the Levi-Civita connections forms, indeed let $\{\theta^i_j\}$ and $\{\psi^{\alpha}_{\beta}\}$ be, respectively, the Levi-Civita connection forms for $(M,g)$ on $\mathcal{U}$ and for $(\mathbb{R}^k,g_{\mbox{can}})$ on $\mathcal{W}$. Then the Levi-Civita connection forms for $(\bar{M},\bar{g})$ on $\overline{\mathcal{U}}$ are given by
		\begin{equation*}
		\bar{\theta}^i_j=\pi_M^*\theta^i_j, \quad \bar{\theta}^i_{m+\alpha}=0, \quad \bar{\theta}^{m+\alpha}_{m+\beta}=\pi^*_{\mathbb{R}^k}\psi^{\alpha}_{\beta}.
		\end{equation*}
		From now on we will omit the pullback from our notation. We collect some well known fact needed in the rest of proof (for a proof see, for instance, Section 4 of \cite{A20} where we deal with semi-Riemannian warped products), that relies on the fact that $\mathbb{R}^k$ is flat and on the definitions \eqref{def di f bar e phi bar} of $\bar{f}$ and $\bar{\varphi}$.
		\begin{itemize}
			\item The non-vanishing components of the Riemann tensor $\overline{\mbox{Riem}}$ of $(\bar{M},\bar{g})$ in the coframe $\{\bar{\theta}^A\}$ are determined by
			\begin{equation}\label{riem tensor per il prodotto comp non nulle}
			\bar{R}_{ijkt}=R_{ijkt},
			\end{equation}
			where $R_{ijkt}$ are the components of the Riemann tensors of $(M,g)$.
			\item The non-vanishing components of $\overline{\mbox{Ric}}$, the Ricci tensor of $(\bar{M},\bar{g})$, in the coframe $\{\bar{\theta}^A\}$ are given by 
			\begin{equation}\label{phi ricci tensor warped prod in terms of u}
			\bar{R}_{ij}=R_{ij},
			\end{equation}
			\item The scalar curvature $\bar{S}$ of $(\bar{M},\bar{g})$ is given by
			\begin{equation}\label{sclar curv del prod riem}
				\bar{S}=S,
			\end{equation}
			where $S$ is the scalar curvature of $(M,g)$.
			\item The components of $\bar{\nabla} \bar{f}$ and of $\overline{\mbox{Hess}}(\bar{f})$ in the coframe $\{\bar{\theta}^A\}$ are given by
			\begin{equation}\label{comp di f in prodotto}
			\bar{f}_i=0, \quad \bar{f}_{m+\alpha}=f_{\alpha}, \quad \bar{f}_{ij}=0, \quad \bar{f}_{i\, m+\alpha}=0, \quad \bar{f}_{ m+\alpha\, m+\beta}=f_{\alpha\beta},
			\end{equation}
			\item The components of $\bar{d}\bar{\varphi}$ and of $\bar{\nabla}\bar{d}\bar{\varphi}$ are given by
			\begin{equation}\label{comp di phi in prodotto}
			\bar{\varphi}^a_i=\varphi^a_i, \quad \bar{\varphi}^a_{m+\alpha}=0, \quad \bar{\varphi}^a_{ij}=\varphi^a_{ij}, \quad \bar{\varphi}^a_{i\, m+\alpha}=0, \quad \bar{\varphi}^a_{m+\alpha\, m+\beta}=0.
			\end{equation}
		\end{itemize}
		
		The equations \eqref{soliton equation for rigid model} in the coframe $\{\bar{\theta}^A\}$ are given by
		\begin{equation}\label{eq soliton prodotto in moving frame}
		\begin{cases}
		\bar{R}_{AB}-\alpha \bar{\varphi}^a_A\bar{\varphi}^a_B+\bar{f}_{AB}=\lambda \delta_{AB}\\
		\bar{\varphi}^a_{AA}=\bar{\varphi}^a_A\bar{f}_A.
		\end{cases}
		\end{equation}
		
		Using the decomposition \eqref{phi ricci tensor warped prod in terms of u} of $\overline{\mbox{Ric}}$, \eqref{comp di phi in prodotto} and \eqref{comp di f in prodotto} the first equation above is equivalent to
		\begin{equation*}
		\begin{dcases}
		R_{ij}-\alpha \varphi^a_i\varphi^a_j=\lambda \delta_{ij}\\
		f_{\alpha \beta}=\lambda \delta_{\alpha \beta}
		\end{dcases}
		\end{equation*}
		while the second is equivalent to
		\begin{equation*}
		\varphi^a_{ii}=0,
		\end{equation*}
		and from our assumption we are able to conclude the validity of \eqref{eq soliton prodotto in moving frame}.
		
		From now on we denote, as usual, by $\bar{R}^{\bar{\varphi}}_{AB}$ the components of $\overline{\mbox{Ric}}^{\bar{\varphi}}=\overline{\mbox{Ric}}-\alpha \bar{\varphi}^*\langle \,,\,\rangle_N$ and by $R^{\varphi}_{ij}$ the components of $\mbox{Ric}^{\varphi}=\mbox{Ric}-\alpha \varphi^*\langle\,,\,\rangle_N$. Combining \eqref{phi ricci tensor warped prod in terms of u} and \eqref{comp di phi in prodotto}, recalling that $(M,g)$ is harmonic-Einstein, the only non-trivial components of $\overline{\mbox{Ric}}^{\bar{\varphi}}$ are given by
		\begin{equation}\label{non triv comp di bar phi ricci in prod}
			\bar{R}^{\bar{\varphi}}_{ij}=R^{\varphi}_{ij}=\lambda \delta_{ij}.
		\end{equation}
		
		Using only that $\mathbb{R}^k$ is flat we easily get
		\begin{equation}\label{der cov di phi ricci in prodotto}
		\bar{\nabla }\overline{\mbox{Ric}}^{\bar{\varphi}}=\pi_M^*\nabla \mbox{Ric}^{\varphi}.
		\end{equation}
		Indeed, by definition of covariant derivative
		\begin{equation*}
		\bar{R}^{\bar{\varphi}}_{AB,C}\bar{\theta}^C=\bar{d}\bar{R}^{\bar{\varphi}}_{AB}-\bar{R}^{\bar{\varphi}}_{CB}\bar{\theta}^C_A-\bar{R}^{\bar{\varphi}}_{AC}\bar{\theta}^C_B,
		\end{equation*}
		that is
		\begin{equation*}
		\bar{R}^{\bar{\varphi}}_{AB,k}\theta^k+\bar{R}^{\bar{\varphi}}_{AB,m+\gamma}\psi^{\gamma}=\bar{d}\bar{R}^{\bar{\varphi}}_{AB}-\bar{R}^{\bar{\varphi}}_{kB}\bar{\theta}^k_A-\bar{R}^{\bar{\varphi}}_{m+\gamma\, B}\bar{\theta}^{m+\gamma}_A-\bar{R}^{\bar{\varphi}}_{A k}\bar{\theta}^k_B-\bar{R}^{\bar{\varphi}}_{A\, m+\gamma}\bar{\theta}^{m+\gamma}_B,
		\end{equation*}
		where we used the obvious notations for the Levi-Civita connection forms. Hence
		\begin{equation*}
		\bar{R}^{\bar{\varphi}}_{ij,k}\theta^k+\bar{R}^{\bar{\varphi}}_{ij,m+\gamma}\psi^{\gamma}=\bar{d}\bar{R}^{\bar{\varphi}}_{ij}-\bar{R}^{\bar{\varphi}}_{kj}\theta^k_i-\bar{R}^{\bar{\varphi}}_{i k}\theta^k_j=R^{\varphi}_{ij,k}\theta^k,
		\end{equation*}
		\begin{equation*}
		\bar{R}^{\bar{\varphi}}_{i\, m+\beta,k}\theta^k+\bar{R}^{\bar{\varphi}}_{i\, m+\beta,m+\gamma}\psi^{\gamma}=\bar{d}\bar{R}^{\bar{\varphi}}_{i \, m+\beta}-\bar{R}^{\bar{\varphi}}_{k\, m+\beta}\theta^k_i-\bar{R}^{\bar{\varphi}}_{i\, m+\gamma}\psi^{\gamma}_{\beta}
		\end{equation*}
		and
		\begin{equation*}
		\bar{R}^{\bar{\varphi}}_{m+\alpha\, m+\beta,k}\theta^k+\bar{R}^{\bar{\varphi}}_{m+\alpha\, m+\beta,m+\gamma}\psi^{\gamma}=\bar{d}\bar{R}^{\bar{\varphi}}_{m+\alpha\, m+\beta}-\bar{R}^{\bar{\varphi}}_{m+\gamma\, m+\beta}\psi^{\gamma}_{\alpha}-\bar{R}^{\bar{\varphi}}_{m+\alpha\, m+\gamma}\psi^{\gamma}_{\beta}=0,
		\end{equation*}
		and thus \eqref{der cov di phi ricci in prodotto} holds.
		
		Now, using also that $(M,g)$ is harmonic-Einstein the above yields $\nabla \mbox{Ric}^{\varphi}=0$ and thus \eqref{nabla phi ricci zero per prodotto} holds.
		
		As pointed out in \hyperref[rmk soliton con phi ricci parallelo]{Remark \ref*{rmk soliton con phi ricci parallelo}} we know that $\bar{S}^{\bar{\varphi}}$ is constant and $\varphi$ is harmonic. Furthermore \eqref{phi bach con phi ricci parallelo} and \eqref{J con phi ricci parallelo} read, respectively,
		\begin{equation*}
			(m+k-2)\bar{B}^{\bar{\varphi}}_{AB}=\bar{W}^{\varphi}_{CADB}\bar{R}^{\bar{\varphi}}_{CD}-\alpha \bar{\varphi}^a_A\bar{\varphi}^a_C\bar{R}^{\bar{\varphi}}_{CB}
		\end{equation*}
		and
		\begin{equation*}
			\bar{J}^a=-2\bar{R}^{\bar{\varphi}}_{AB}\bar{\varphi}^a_{AB}.
		\end{equation*}
		
		Using \eqref{comp di phi in prodotto} and the fact that $(M,g)$ is harmonic-Einstein we infer
		\begin{equation*}
			\bar{J}^a=-2R^{\varphi}_{ij}\varphi^a_{ij}=-2\lambda\tau(\varphi)^a=0,
		\end{equation*}
		hence \eqref{bar J per modello rigid} holds.
		
		Using \eqref{comp di phi in prodotto} and \eqref{non triv comp di bar phi ricci in prod} the components of the $\bar{\varphi}$-Bach tensor are given by
		\begin{equation}\label{bar phi bach con A uguale a i}
		(m+k-2)\bar{B}^{\bar{\varphi}}_{AB}=\lambda(\bar{W}^{\varphi}_{kAkB}-\alpha \varphi^a_i\varphi^a_k\delta_{iA}\delta_{kB})
		\end{equation}
		To compute them we need the components of the $\bar{\varphi}$-Weyl tensor and to compute them we need the components of the $\bar{\varphi}$-Schouten tensor.
		
		By definition the $\bar{\varphi}$-Schouten tensor has components
		\begin{equation}\label{phi bar schouten def generale delle suo comp}
			\bar{A}^{\bar{\varphi}}_{AB}=\bar{R}^{\bar{\varphi}}_{AB}-\frac{\bar{S}^{\bar{\varphi}}}{2(m+k-1)}\delta_{AB}.
		\end{equation}
		Since $\mathbb{R}^k$ is flat, $\bar{\varphi}=\varphi\circ \pi_M$ and $(M,g)$ is harmonic-Einstein, we have $\bar{S}^{\bar{\varphi}}=S^{\varphi}=m\lambda$ and thus the above can be written as
		\begin{equation*}
			\bar{A}^{\bar{\varphi}}_{AB}=\bar{R}^{\bar{\varphi}}_{AB}-\frac{m\lambda}{2(m+k-1)}\delta_{AB}.
		\end{equation*}
		From \eqref{phi bar schouten def generale delle suo comp}, using that $R^{\varphi}_{ij}=\lambda\delta_{ij}$, the only non-trivial components of the $\bar{\varphi}$-Schouten tensor are given by
		\begin{equation}\label{phi bar schouten def generale delle suo comp i e j}
		\bar{A}^{\bar{\varphi}}_{ij}=\frac{m-2+2k}{2(m+k-1)}\lambda\delta_{ij}
		\end{equation}
		and
		\begin{equation}\label{phi bar schouten def generale delle suo comp alpha e beta}
		\bar{A}^{\bar{\varphi}}_{m+\alpha\,m+\beta}=-\frac{m}{2(m+k-1)}\lambda\delta_{\alpha\beta}
		\end{equation}
		
		By definition, see \eqref{comp phi weyl}, the components of the $\bar{\varphi}$-Weyl tensor are given by
		\begin{equation*}
			\bar{W}^{\varphi}_{ABCD}=\bar{R}_{ABCD}-\frac{1}{m+k-2}(\bar{A}^{\bar{\varphi}}_{AC}\delta_{BD}-\bar{A}^{\bar{\varphi}}_{AD}\delta_{BC}+\bar{A}^{\bar{\varphi}}_{BD}\delta_{AC}-\bar{A}^{\bar{\varphi}}_{BC}\delta_{AD}).
		\end{equation*}
		Using \eqref{riem tensor per il prodotto comp non nulle}, \eqref{phi bar schouten def generale delle suo comp i e j} and \eqref{phi bar schouten def generale delle suo comp alpha e beta} from the above we get
		\begin{equation}\label{comp phi weyl t i j k}
		\bar{W}^{\bar{\varphi}}_{tijk}=R_{tijk}-\frac{m-2+2k}{(m+k-1)(m+k-2)}\lambda(\delta_{tj}\delta_{ik}-\delta_{tk}\delta_{ij}),
		\end{equation}
		\begin{equation}\label{comp phi weyl t i j  e m + alpfa}
		\bar{W}^{\bar{\varphi}}_{ti\, m+\alpha\, k}=0,
		\end{equation}
		and
		\begin{equation}\label{comp phi weyl t i m alpa e m beta}
		\bar{W}^{\bar{\varphi}}_{m+\beta i\, m+\alpha\, k}=-\frac{k-1}{(m+k-1)(m+k-2)}\lambda\delta_{\alpha\beta}\delta_{ik}
		\end{equation}

		From \eqref{bar phi bach con A uguale a i} with $A=i$ and $B=j$ we have
		\begin{equation}\label{prime comp di phi Bach}
		(m+k-2)\bar{B}^{\bar{\varphi}}_{ij}=\lambda(\bar{W}^{\varphi}_{kikj}-\alpha \varphi^a_i\varphi^a_j)
		\end{equation}
		Tracing \eqref{comp phi weyl t i j k} we obtain
		\begin{equation*}
		\bar{W}^{\bar{\varphi}}_{titk}=R_{ik}-\frac{(m-2+2k)(m-1)}{(m+k-1)(m+k-2)}\lambda\delta_{ik},
		\end{equation*}
		hence, by definition of $\varphi$-Ricci tensor and since $(M,g)$ is harmonic-Einstein,
		\begin{align*}
		\bar{W}^{\varphi}_{kikj}-\alpha \varphi^a_i\varphi^a_j=\frac{k(k-1)}{(m+k-1)(m+k-2)}\lambda\delta_{ij}.
		\end{align*}
		Plugging the above into \eqref{prime comp di phi Bach} we infer
		\begin{equation}\label{comp phi bach prod i e j}
		(m+k-2)\bar{B}^{\bar{\varphi}}_{ij}=\frac{k(k-1)}{(m+k-1)(m+k-2)}\lambda^2\delta_{ij},
		\end{equation}
		
		From \eqref{bar phi bach con A uguale a i} with $A=i$ and $B=m+\alpha$, using \eqref{comp phi weyl t i j  e m + alpfa} we get
		\begin{equation}\label{comp phi bach prod i e m alpa}
		(m+k-2)\bar{B}^{\bar{\varphi}}_{i\,m+\alpha}=\lambda\bar{W}^{\varphi}_{kik\,m+\alpha}=0.
		\end{equation}
		
		Finally, From \eqref{bar phi bach con A uguale a i} with $A=m+\alpha$ and $B=m+\beta$
		\begin{equation}
		(m+k-2)\bar{B}^{\bar{\varphi}}_{m+\alpha\, m+\beta}=\lambda\bar{W}^{\varphi}_{k \, m+\alpha\, k\, m+\beta}.
		\end{equation}
		Tracing \eqref{comp phi weyl t i m alpa e m beta} we infer
		\begin{equation}
		\bar{W}^{\bar{\varphi}}_{m+\beta k\, m+\alpha\, k}=-\frac{m(k-1)}{(m+k-1)(m+k-2)}\lambda\delta_{\alpha\beta}
		\end{equation}
		and thus
		\begin{equation}\label{comp phi bach prod m alpa e m beta}
		(m+k-2)\bar{B}^{\bar{\varphi}}_{m+\alpha \, m+\beta}=-\frac{m(k-1)}{(m+k-1)(m+k-2)}\lambda^2\delta_{\alpha\beta}.
		\end{equation}
		
		Combining \eqref{comp phi bach prod i e j}, \eqref{comp phi bach prod i e m alpa} and \eqref{comp phi bach prod m alpa e m beta} we conclude the validity of \eqref{bar phi bach per modello rigid}.
	\end{proof}
	
	In the following Theorem we show that having parallel $\varphi$-Ricci tensor is a condition that characterizes rigidity for complete non-compact non-steady gradient harmonic-Ricci soliton, see \hyperref[rmk ipotesi petersen wylei]{Remark \ref*{rmk ipotesi petersen wylei}} below for further comments on the assumptions. We take care of the steady case in \hyperref[prop caratt caso steady]{Proposition \ref*{prop caratt caso steady}} below.
	\begin{thm}\label{thm rigidity}
		Let $M$ be a complete non-compact non-steady gradient harmonic-Ricci soliton, that is,
		\begin{equation}\label{grad ricci harm sol per dim rigid}
		\begin{dcases}
		\mbox{Ric}-\alpha \varphi^*\langle\,,\,\rangle_N+\mbox{Hess}(f)=\lambda g\\
		\tau(\varphi)=d\varphi(\nabla f)
		\end{dcases}
		\end{equation}
		holds for some positive constant $\alpha$, $\varphi:M\to N$ a smooth map, where $(N,\langle\,,\,\rangle_N)$ is a Riemannian manifold, $f\in\mathcal{C}^{\infty}(M)$ and $\lambda\in\mathbb{R}\setminus \{0\}$. Assume that the $\varphi$-Ricci tensor is parallel. Then \eqref{grad ricci harm sol per dim rigid} is rigid.
	\end{thm}
	\begin{proof}
		Eventually after passing to the universal cover $(\tilde{M},\tilde{g})$ of $(M,g)$ we can assume that $(M,g)$ is simply connected. Indeed, since the projection $p:(\tilde{M},\tilde{g})\to (M,g)$ is a local isometry, $(\tilde{M},\tilde{g})$ is itself a Ricci-harmonic soliton with respect to $\alpha$, $\lambda$ and the lifting of $f$ and $\varphi$ to $\widetilde{M}$.
		
		Since the $\varphi$-Ricci tensor is parallel, with the aid of \hyperref[rmk soliton con phi ricci parallelo]{Remark \ref*{rmk soliton con phi ricci parallelo}} we know that $\varphi$ is harmonic, hence $d\varphi(\nabla f)=0=\tau(\varphi)$, and the $\varphi$-scalar curvature is constant.
		
		Recall that $f$ satisfies the Hamilton-type identity \eqref{ham type id}. Hence, since $S^{\varphi}$ is constant,
		\begin{equation}\label{hess f per via dell id di hamiton}
		\mbox{Hess}(|\nabla f|^2)=2\lambda \mbox{Hess}(f).
		\end{equation}
		
		Notice that since $\mbox{Ric}^{\varphi}$ is parallel using the first equation of \eqref{grad ricci harm sol per dim rigid} we infer that also $\mbox{Hess}(f)$ is parallel, that is, $f_{ijk}=0$ in a local orthonormal coframe. Then we easily get
		\begin{equation*}
			\mbox{Hess}(|\nabla f|^2)=2\mbox{Hess}(f)^2.
		\end{equation*}
		Indeed, in a local orthonormal coframe
		\begin{equation*}
		|\nabla f|^2_{ij}=2(f_{kij}f_k+f_{ik}f_{kj}).
		\end{equation*}
		Then \eqref{hess f per via dell id di hamiton} reads
		\begin{equation*}
		\lambda \mbox{Hess}(f)=\mbox{Hess}(f)^2,
		\end{equation*}
		
		From the above equation the only possible eigenvalues of $\mbox{Hess}(f)$ are $0$ and $\lambda$. Using the first equation of \eqref{grad ricci harm sol per dim rigid} the same holds also for $\mbox{Ric}^{\varphi}$.
		
		If $f$ is constant then $(M,g)$ is harmonic-Einstein, hence rigid. From now on we assume that $f$ is non-constant on $M$.
		
		Since $S^{\varphi}$ is constant and $f$ is non-constant, from \eqref{form gradiente phi curv scalare in solitoni gradiente} we see that the $\varphi$-Ricci tensor has always $0$ as eigenvalue and thus the Hessian of $f$ has always $\lambda$ as eigenvalue, at least at a point $p\in M$.
		
		We denote by $V_0$ and $V_{\lambda}$ the eigenspace bundles of $\mbox{Hess}(f)$. Since $M$ is connected and the eigenvalues of $\mbox{Hess}(f)$ are smooth we know that $TM$ splits orthogonally as $V_0\oplus V_{\lambda}$. Furthermore, since we are assuming $(M,g)$ simply connected, $V_0$ and $V_{\lambda}$ are parallel and integrable distributions and $(M,g)$ splits as the Riemannian product of two complete totally geodesic submanifolds $(L,g_L)$ and $(E,g_E)$ such that $TE=\left.V_{\lambda}\right|_E$ and $TL=\left.V_0\right|_L$. Moreover, via the splitting,
		\begin{equation}\label{hess di f per via dello splitting}
			\mbox{Hess}(f)=\lambda \pi^*_Eg_E
		\end{equation}
		and
		\begin{equation}\label{phi ricci per via dello splitting}
			\mbox{Ric}^{\varphi}=\lambda \pi^*_Lg_L.
		\end{equation}
		This is essentially the de Rham splitting theorem \cite{dR}, see for instance Section 3.2 of \cite{J} (or also 2.3 at page 55 of \cite{DMVVZ}). We denote by $k$ the dimension of $E$, we know that $k\geq 1$.
		
		The validity of \eqref{hess di f per via dello splitting} implies that, via the splitting, $f=f_E\circ \pi_E$ for some $f_E\in\mathcal{C}^{\infty}(E)$ such that
		\begin{equation*}
			{}^E\mbox{Hess}(f_E)=\lambda g_E,
		\end{equation*}
		on the complete Riemannian manifold $(E,g_E)$. Then, via the classic theorem of Y. Tashiro \cite{T} (see also Theorem 8.5 of \cite{AMR}), since $\lambda\neq 0$ we deduce that $(E,g_E)$ is isometric to $\mathbb{R}^k$ with the Euclidean metric and, via the isometry, $f_E(x)=\frac{\lambda}{2}|x|^2+\langle b,x\rangle+c$ for some $b\in\mathbb{R}^k$ and $c\in\mathbb{R}$. Furthermore, \eqref{phi ricci per via dello splitting} gives that $\pi^*_E\mbox{Ric}^{\varphi}=\pi^*_E\mbox{Ric}-\alpha \pi^*_E\varphi^*\langle\,,\,\rangle_N$ vanishes on $E$. Via the splitting $\pi^*_E\mbox{Ric}={}^E\mbox{Ric}=0$, since $E$ is flat, hence $\pi^*_E\varphi^*\langle\,,\,\rangle_N=(\varphi\circ \pi_E)^*\langle\,,\,\rangle_N$ vanishes on $E$. Then $\varphi\circ \pi_E$ is constant. Since the argument above applies for every leaf of the foliation $E\times L$ of $M$ we deduce that $\varphi=\varphi_L\circ \pi_L$, for some smooth map $\varphi_L:L\to N$. Since $\varphi$ is harmonic and $\varphi=\varphi_L\circ \pi_L$ we immediately get, via \eqref{comp di phi in prodotto}, that $\varphi_L$ is harmonic too. Now, using \eqref{phi ricci per via dello splitting} we get
		\begin{equation*}
		{}^L\mbox{Ric}-\alpha \varphi_L^*\langle \,,\,\rangle_N=\lambda g_L
		\end{equation*}
		and since $\varphi_L$ is harmonic we conclude that $(L,g_L)$ is harmonic-Einstein with respect to $\varphi_L$, $\alpha$ and $\lambda$.
	\end{proof}
	\begin{rmk}\label{rmk thm rigidity per shrinking}
		By the proof of the Theorem above we know that the universal cover $(\tilde{M},\tilde{g})$ of $(M,g)$ is isometric to the Riemannian product of $L\times \mathbb{R}^k$ described in \hyperref[prop che descrive modello rigido]{Proposition \ref*{prop che descrive modello rigido}}, for some $0\leq k\leq m$.
		
		When the soliton is shrinking, i.e., when $\lambda>0$, we know that $L$ is compact. Indeed, using that $\alpha>0$
		\begin{equation*}
		{}^L\mbox{Ric}\geq {}^L\mbox{Ric}-\alpha \varphi_L^*\langle \,,\,\rangle_N=\lambda g_L,
		\end{equation*}
		hence we can apply Myers's theorem. Furthermore, since
		\begin{equation*}
		\mbox{Ric}_f\geq \mbox{Ric}-\alpha \varphi^*\langle\,,\,\rangle_N+\mbox{Hess}(f)=\lambda g,
		\end{equation*}
		using the weighted version of the Myers's theorem holds (see Theorem 1 of \cite{N}) we obtain that the fundamental group of $M$ is finite.
		
		Then, for shrinking solitons, the theorem above can be strengthen, obtaining that $(M,g)$ is isometric to a finite quotient of $L\times \mathbb{R}^k$ for some $1\leq k\leq m$ ($k=0$ is excluded because $M$ is non-compact), where $(L,g_L)$ is a compact harmonic-Einstein manifold with respect to $\alpha$, $\varphi$ and $\lambda$ and, via the isometry, $\varphi=\varphi_L\circ \pi_L$ and $f=f_{\mathbb{R}^k}\circ \pi_{\mathbb{R}^k}$, where $f_{\mathbb{R}^k}=\frac{\lambda}{2}|x|^2+\langle b,x\rangle+c$ for some $c\in\mathbb{R}$ and $b\in\mathbb{R}^k$ and $\pi_L:L\times \mathbb{R}^k\to L$ and where $\pi_{\mathbb{R}^k}:L\times \mathbb{R}^k\to \mathbb{R}^k$ are the canonical projections.
	\end{rmk}
	
	\begin{rmk}\label{rmk ipotesi petersen wylei}
		In Theorem 1.2 of \cite{PW} the authors characterized rigidity for complete non-compact non-steady Ricci soliton via the constancy of the scalar curvature and the radial flatness, that is,
		\begin{equation*}
			R(\cdot,\nabla f )\nabla f=0.
		\end{equation*}
		
		We could have done the analogous here, i.e., characterize rigidity for complete non-compact non-steady harmonic-Ricci soliton via the constancy of the $\varphi$-scalar curvature and the radial flatness, but we preferred the assumption of parallel $\varphi$-Ricci tensor. The reason for this choice are two. The first is that, as seen in \hyperref[prop che descrive modello rigido]{Proposition \ref*{prop che descrive modello rigido}}, rigid harmonic-Ricci soliton have actually parallel $\varphi$-Ricci tensor. The second is that the proof with this assumption is simpler, more similar to the one of Theorem 2 of \cite{N} than the one of Theorem 1.2 of \cite{PW}, and for our purpose this characterization will be sufficient.
		
		Notice that for a gradient harmonic-Ricci soliton having parallel $\varphi$-Ricci tensor implies the constancy of the $\varphi$-scalar curvature, as seen in \hyperref[rmk soliton con phi ricci parallelo]{Remark \ref*{rmk soliton con phi ricci parallelo}}, and the radial flatness, since \eqref{form commutazione der covariante per ricci phi in solitoni gradiente} gives $f_tR_{tijk}=0$, that is, $R(\cdot,\cdot)\nabla f=0$.
	\end{rmk}
	
	To conclude this Section we deal with complete gradient steady harmonic-Ricci soliton.
	\begin{prop}\label{prop caratt caso steady}
		Let $(M,g)$ be a complete gradient steady harmonic-Ricci soliton of dimension $m\geq 2$ with respect to a positive constant $\alpha$, a smooth map $\varphi:M\to N$, where $(N,\langle\,,\,\rangle_N)$ is a Riemannian manifold, and $f\in\mathcal{C}^{\infty}(M)$. If the $\varphi$-scalar curvature is constant and $f$ is non-constant then $(M,g)$ is isometric to the the Riemannian product of $\mathbb{R}$ with a totally geodesic $\psi$-Ricci flat (with respect to $\alpha$) hypersurface $\Sigma$, where $\psi:=\left.\varphi\right|_{\Sigma}$. Moreover $\varphi=\psi\circ \pi_{\Sigma}$ on $\mathbb{R}\times \Sigma$, where $\pi_{\Sigma}:\mathbb{R}\times \Sigma\to \Sigma$ is the canonical projection and the function $f$ can be expressed on $\mathbb{R}\times \Sigma$ as
		\begin{equation*}
			f(t,x)=bt+c \quad \mbox{ for every } t\in \mathbb{R} \mbox{ and } x\in\Sigma,
		\end{equation*}
		for some $b>0$ and $c\in\mathbb{R}$ such that $\Sigma=f^{-1}(\{b\})$.
		
		In particular, complete gradient steady harmonic-Ricci solitons with constant $\varphi$-scalar curvature are rigid.
	\end{prop}
	\begin{proof}
		It is an easy application of $i)$ of \hyperref[thm stime phi curv scalare nella tesi]{Theorem \ref*{thm stime phi curv scalare nella tesi}}. Indeed, since $\lambda=0$ we have that $S^{\varphi}_*=0$. Since we assumed the constancy of the $\varphi$-scalar curvature we have $S^{\varphi}=S^{\varphi}_*=0$ on $M$, hence the thesis.
	\end{proof}
	
	\begin{rmk}
		\hyperref[prop caratt caso steady]{Proposition \ref*{prop caratt caso steady}} is apparently in contrast with \hyperref[prop che descrive modello rigido]{Proposition \ref*{prop che descrive modello rigido}}, where we showed that the Riemannian product $\bar{M}$ of any complete $\varphi$-Ricci flat manifold with $\mathbb{R}^k$, for any non-negative integer $k$ (and not only for $k=1$), produces a complete gradient steady harmonic-Ricci soliton with constant $\bar{\varphi}$-scalar curvature.
		
		To overcome this contrast notice that, when $\lambda=0$, we can always reduce to the cases where $k=0,1$. Indeed, when $k\geq 2$, the affine function $f$ is given by
		\begin{equation*}
		f:\mathbb{R}^k\to \mathbb{R}, \quad x\mapsto f(x)=\langle x,b\rangle+c,
		\end{equation*}
		for some $b\in\mathbb{R}^k$ and $c\in\mathbb{R}$. If $b=0$ then $f$ is constant, hence we can reduce to the case $k=0$ by replacing $M$ with $M\times \mathbb{R}^k$, that is $\varphi$-Ricci flat with respect to the trivial lifting of $\varphi$. Assume now that $b\neq 0$. Then we may choose a orthonormal basis $\{v_1,\ldots ,v_k\}$ of $\mathbb{R}^k$ such that
		\begin{equation*}
		v_k=\frac{b}{|b|}.
		\end{equation*}
		Then, letting $x=x^1v_1+\ldots +x^kv_k$, we see that $f(x)=|b|x^k+c$. Hence, via the following isometry
		\begin{equation*}
		\psi:\mathbb{R}^k\to \mathbb{R}^{k-1}\times \mathbb{R}, \quad x\mapsto (x-x^kv_k,|b|x_k+c),
		\end{equation*}
		the map $f$ takes the form
		\begin{equation*}
		f(y,r)=|b|r+c.
		\end{equation*}
		Then $\bar{M}=M\times \mathbb{R}^k$ can be seen as the Riemannian product between the $\varphi$-Ricci flat manifold $\tilde{M}:=M\times \mathbb{R}^{k-1}$ (with respect to the trivial lifting of $\varphi$ to $M\times \mathbb{R}^k$) and $\mathbb{R}$ while $\bar{f}$ can be seen as the lifting of $\tilde{f}(x)=|b|x+c$, for $x\in \mathbb{R}$, to $\bar{M}=\tilde{M}\times \mathbb{R}$.
	\end{rmk}

	\section{Compact harmonic-Ricci solitons}\label{section solitoni compatti}
	
	In this section we consider compact harmonic-Ricci solitons of dimension $m\geq 2$
	\begin{equation}\label{harmonic ricci soliton euqation}
	\begin{dcases}
	\mbox{Ric}^{\varphi}+\frac{1}{2}\mathcal{L}_Xg=\lambda g\\
	\tau(\varphi)=d\varphi(X)
	\end{dcases}
	\end{equation}
	for $X\in\mathfrak{X}(M)$, $\lambda\in\mathbb{R}$ and $\alpha>0$. In \hyperref[sub rid a sol bgrad]{Section \ref*{sub rid a sol bgrad}} we show that if \eqref{harmonic ricci soliton euqation} is not rigid then it is gradient and shrinking. Using this important information, in \hyperref[sect rig caso compatto phi cotton]{Section \ref*{sect rig caso compatto phi cotton}}, we prove that the soliton \eqref{harmonic ricci soliton euqation} is rigid if and only if it is $\varphi$-Cotton flat (see \hyperref[thm phi cotton flat compact ricci solitons are rigid]{Theorem \ref*{thm phi cotton flat compact ricci solitons are rigid}} below).
	
	\subsection{Reduction to compact gradient shrinking harmonic-Ricci solitons}\label{sub rid a sol bgrad}
	
	\begin{rmk}\label{rmk rigido in caso compatto se e solo se phi curv scalar cost}
		Using \hyperref[thm with S phi constant is the manifold is a sphere with mu equal to zero]{Theorem \ref*{thm with S phi constant is the manifold is a sphere with mu equal to zero}} it is easy to see that the harmonic-Ricci soliton is rigid (i.e., is harmonic-Einstein) if and only if $S^{\varphi}$ is constant.
	\end{rmk}
	
	\begin{prop}\label{prop compact ricci harm soliton non rigid shrink}
		If the compact harmonic-Ricci soliton \eqref{harmonic ricci soliton euqation} is not rigid then it is shrinking.
	\end{prop}
	\begin{proof}
		Taking the trace of the first equation of \eqref{harmonic ricci soliton euqation} we get $S^{\varphi}+\mbox{div}(X)=m\lambda$, hence, using the divergence theorem
		\begin{equation*}
		\int_MS^{\varphi}=m\lambda \mbox{vol}(M).
		\end{equation*}
		By setting
		\begin{equation*}
		\bar{S}^{\varphi}:=\frac{1}{\mbox{vol}(M)}\int_MS^{\varphi}, \quad S^{\varphi}_*:=\min_MS^{\varphi},
		\end{equation*}
		assuming that the soliton is not rigid, from \hyperref[rmk rigido in caso compatto se e solo se phi curv scalar cost]{Remark \ref*{rmk rigido in caso compatto se e solo se phi curv scalar cost}} cannot be constant and thus
		\begin{equation}\label{inf phi curv scalare minore di m lambda}
		S^{\varphi}_*<\bar{S}^{\varphi}=m\lambda.
		\end{equation}
		
		On the other hand equation \eqref{X laplacian per strutt tipo Einstein} gives
		\begin{equation*}
		\frac{1}{2}\Delta_XS^{\varphi}+|\mathring{\mbox{Ric}}^{\varphi}|^2+\alpha|\tau(\varphi)|^2+\frac{S^{\varphi}}{m}(S^{\varphi}-m\lambda)=0,
		\end{equation*}
		hence
		\begin{equation*}
		\frac{1}{2}\Delta_XS^{\varphi}+\frac{S^{\varphi}}{m}(S^{\varphi}-\bar{S}^{\varphi})\leq 0.
		\end{equation*}
		Let $x_*\in M$ such that $S^{\varphi}(x_*)=S^{\varphi}_*$. Using that $\Delta_XS^{\varphi}(x^*)\leq 0$, evaluating the above at $x_*$ we get
		\begin{equation*}
		S^{\varphi}_*(S^{\varphi}_*-\bar{S}^{\varphi})\leq 0.
		\end{equation*}
		Since the harmonic-Ricci soliton is not rigid we have $S^{\varphi}_*-\bar{S}^{\varphi}<0$, then from the above we infer $S^{\varphi}_*\geq 0$ and combining it with \eqref{inf phi curv scalare minore di m lambda} we obtain
		\begin{equation*}
		0\leq S^{\varphi}_*<m\lambda,
		\end{equation*}
		and thus $\lambda>0$.
	\end{proof}
	
	\begin{comment}
		\begin{rmk}
		Notice that the compact a harmonic-Ricci soliton \eqref{harmonic ricci soliton euqation} is rigid if and only if $X$ is a vertical Killing vector field and, if this is the case, $S^{\varphi}=m\lambda$. Indeed, if $X$ is a vertical Killing vector field from \eqref{harmonic ricci soliton euqation} we immediately get
		\begin{equation*}
		\begin{dcases}
		\mbox{Ric}^{\varphi}=\lambda g\\
		\tau(\varphi)=0,
		\end{dcases}
		\end{equation*}
		hence $(M,g)$ is harmonic-Einstein with scalar curvature $m\lambda$. Conversely, if $(M,g)$ is harmonic-Einstein then $X$ is a homothetic vertical vector field, in particular $\mbox{div}(X)=m\lambda-S^{\varphi}$ is constant. Since $M$ is compact we conclude that $\mbox{div}(X)=0$, that is, $S^{\varphi}=m\lambda$, and then $X$ is Killing.
		
		Using \hyperref[prop vertical killing vector field]{Proposition \ref*{prop vertical killing vector field}}, for a compact expanding harmonic-Ricci soliton the rigidity is equivalent to $X=0$ while for a compact steady harmonic-Ricci soliton the rigidity is equivalent to $\nabla X=0$.
		
		Furthermore, since on a compact Riemannian manifold the affine functions are only the constant functions, a compact gradient harmonic-Ricci soliton is rigid if and only if its potential function is constant.
		\end{rmk}
	\end{comment}
	
	Combining Theorem 1.1 of \cite{YZ} with \hyperref[prop compact ricci harm soliton non rigid shrink]{Proposition \ref*{prop compact ricci harm soliton non rigid shrink}} we get the following
	\begin{thm}\label{thm solitone ricci harm compatto non rigido grad shrink}
		If a compact harmonic-Ricci soliton of dimension $m\geq 2$ is not rigid then it is gradient and shrinking.
	\end{thm}
	
	In the rest of this Section, for completeness and for the reader convenience, we provide some details regarding the proof of Theorem 1.1 of \cite{YZ}.
	\begin{rmk}
		The key point of Theorem 1.1 of \cite{YZ} is to guarantee the existence of a smooth function $f$ on a compact shrinking harmonic-Ricci soliton $(M,g)$ such that
		\begin{equation}\label{eq soddisfatta da muller perelman potential}
		S^{\varphi}+2\Delta f-|\nabla f|^2+2\lambda f \quad \mbox{ is constant on } M.
		\end{equation}
		
		Notice that the existence of such function $f$ had been established by R. M\"{u}ller in Section 7.3 of \cite{M}, where he proved the existence on any compact Riemannian manifold $(M,\widetilde{g})$ of a unique $v\in\mathcal{C}^{\infty}(M)$ positive normalized eigenvector of
		\begin{equation*}
		L(v)=-4\widetilde{\Delta} v+\left(\widetilde{S}^{\varphi}-\frac{m}{2}\log(4\pi)-m\right)v-2v\log v,
		\end{equation*}
		where we denoted by $\widetilde{\Delta}$ and $\widetilde{S}^{\varphi}$, respectively, the Laplacian and the $\varphi$-scalar curvature evaluated with respect to the metric $\widetilde{g}$. By setting $v=e^{-\frac{f}{2}}$ the above yields the constancy on $M$ of
		\begin{equation*}
		\widetilde{S}^{\varphi}+2\widetilde{\Delta} f-|\widetilde{\nabla} f|_{\widetilde{g}}^2+f.
		\end{equation*}
		By setting $\widetilde{g}:=\beta^2g$, for any positive constant $\beta$, the above gives the constancy of
		\begin{equation*}
		S^{\varphi}+2\Delta f-|\nabla f|^2+\beta^2f.
		\end{equation*}
		To obtain \eqref{eq soddisfatta da muller perelman potential} it is sufficient to choose $\beta^2=2\lambda$, and this can be done since the harmonic-Ricci soliton is shrinking.
	\end{rmk}
	\begin{dfn}
		We call the unique (up to an additive constant) function $f$ such that \eqref{eq soddisfatta da muller perelman potential} holds the {\em M\"{u}ller-Perelman potential}.
	\end{dfn}
	\begin{rmk}
		Once we have the existence of the M\"{u}ller-Perelman potential on a compact shrinking harmonic-Ricci soliton $(M,g)$ we can prove Theorem 1.1 of \cite{YZ} relying on the validity of
		\begin{equation}\label{form derdinzky}
		\int_M[|h_1|^2-\langle h_1,b_1\rangle+\alpha|h_2|^2-\alpha\langle h_2,b_2\rangle]e^{-f}=\frac{1}{2}\int_M\left(S^{\varphi}+2\Delta f-|\nabla f|^2+2\lambda f\right)\mbox{div}(e^{-f}Y),
		\end{equation}
		where
		\begin{equation}\label{def di Y per thm derd}
			Y:=\nabla h-X,
		\end{equation}
		\begin{equation}\label{def h 1 e h 2}
		h_1:=\mbox{Ric}^{\varphi}+\mbox{Hess}(f)-\lambda g, \quad h_2:=\tau(\varphi)-d\varphi(\nabla f)
		\end{equation}
		and
		\begin{equation}\label{def di b1 e b 2}
		b_1:=\mbox{Ric}^{\varphi}+\frac{1}{2}\mathcal{L}_Xg-\lambda g, \quad h_2:=\tau(\varphi)-d\varphi(X),
		\end{equation}
		whose proof (that is computational) is postponed to \hyperref[prop form derd]{Proposition \ref*{prop form derd}} below. Indeed, let $(M,g)$ be a compact shrinking harmonic-Ricci soliton with potential $X$ and let $f$ be the M\"{u}ller-Perelman potential. Then $b_1=0=b_2$ and thus \eqref{form derdinzky} gives, using \eqref{eq soddisfatta da muller perelman potential},
		\begin{equation*}
		\int_M(|h_1|^2+\alpha|h_2|^2)e^{-f}=0.
		\end{equation*}
		Hence, since $\alpha>0$ we have $h_1=0=h_2$, that is,
		\begin{equation*}
		\begin{dcases}
		\mbox{Ric}^{\varphi}+\mbox{Hess}(f)=\lambda g\\
		\tau(\varphi)=d\varphi(\nabla f).
		\end{dcases}
		\end{equation*}
		The above shows the compact shrinking harmonic-Ricci soliton is gradient.
	\end{rmk}
	
	\begin{rmk}
		Recall that on any compact Riemannian manifold $(M,g)$ every vector field $X\in\mathfrak{X}(M)$ can be decomposed as
		\begin{equation*}
			X=\nabla h+Y,
		\end{equation*}
		where $Y$ is a divergence free vector field and the smooth function $h\in\mathcal{C}^{\infty}(M)$, defined up to an additive constant, is called Hodge-de Rham potential of $X$, see for instance \cite{ABR}.
		
		As pointed out in Proposition 2.1 of \cite{YZ}, the M\"{u}ller-Perelman potential coincide with the Hodge-de Rham potential, since $f-h$ is a harmonic function.
	\end{rmk}
		 
	\begin{prop}\label{prop form derd}
		For every compact Riemannian manifold $(M,g)$, $X\in\mathfrak{X}(M)$ and $f\in\mathcal{C}^{\infty}(M)$, by setting $Y$ as in \eqref{def di Y per thm derd}, $h_1$ and $h_2$ as in \eqref{def h 1 e h 2} and $b_1$ and $b_2$ as in \eqref{def di b1 e b 2}, formula \eqref{form derdinzky} holds.
	\end{prop}
	\begin{proof}
		We claim the validity of
		\begin{equation}\label{id per mostr form derd}
		\begin{aligned}
		\int_M&\left[\left\langle \mbox{Ric}^{\varphi}+\mbox{Hess}(f),\frac{1}{2}\mathcal{L}_Yg\right\rangle+\alpha\langle \tau(\varphi)-d\varphi(\nabla f),d\varphi(Y)\rangle\right] e^{-f}\\
		&=\int_M\left(\frac{1}{2}S^{\varphi}+\Delta f-\frac{1}{2}|\nabla f|^2\right)\mbox{div}(e^{-f}Y).
		\end{aligned}
		\end{equation}
		We have, in a local orthonormal coframe,
		\begin{equation}\label{dim form derd}
		\begin{aligned}
		\left\langle \mbox{Ric}^{\varphi}+\mbox{Hess}(f),\frac{1}{2}\mathcal{L}_Yg\right\rangle e^{-f}=&(R^{\varphi}_{ij}+f_{ij})Y^i_je^{-f}\\
		=&[(R^{\varphi}_{ij}+f_{ij})Y^ie^{-f}]_j-(R^{\varphi}_{ij}+f_{ij})_jY^ie^{-f}+(R^{\varphi}_{ij}+f_{ij})f_jY^ie^{-f}.
		\end{aligned}
		\end{equation}
		Using Schur's lemma \eqref{div of phi Ricci}, the commutation relation \eqref{comm rule der terza funzione} and the definition of the $\varphi$-Ricci tensor
		\begin{align*}
		(R^{\varphi}_{ij}+f_{ij})_j=&\frac{1}{2}S^{\varphi}_i-\alpha\varphi^a_{jj}\varphi^a_i+f_{jij}=\frac{1}{2}S^{\varphi}_i-\alpha\varphi^a_{jj}\varphi^a_i+(\Delta f)_i+R_{ij}f_j\\
		=&\left(\frac{1}{2}S^{\varphi}+\Delta f\right)_i-\alpha\varphi^a_{jj}\varphi^a_i+R_{ij}^{\varphi}f_j+\alpha\varphi^a_i\varphi^a_jf_j,
		\end{align*}
		hence we easily get
		\begin{equation}\label{primo pezzo dim derd}
		\begin{aligned}
		(R^{\varphi}_{ij}+f_{ij})_jY^ie^{-f}=&\left[\left(\frac{1}{2}S^{\varphi}+\Delta f\right)e^{-f}Y^i\right]_i-\left(\frac{1}{2}S^{\varphi}+\Delta f\right)(e^{-f}Y^i)_i\\
		&+R^{\varphi}_{ij}f_jY^i-\alpha(\varphi^a_{jj}-\varphi^a_jf_j)\varphi^a_iY^ie^{-f}.
		\end{aligned}
		\end{equation}
		Furthermore
		\begin{equation}\label{secondo pezzo dim derd}
		f_{ij}f_jY^ie^{-f}=\frac{1}{2}|\nabla f|^2_iY^ie^{-f}=\frac{1}{2}\left(|\nabla f|^2e^{-f}Y^i\right)_i-\frac{1}{2}|\nabla f|^2\mbox{div}(e^{-f}Y).
		\end{equation}
		The claim follows by plugging \eqref{primo pezzo dim derd} and \eqref{secondo pezzo dim derd} into \eqref{dim form derd}, using the divergence theorem and rearranging the terms.
		
		Clearly have
		\begin{equation*}
		\langle h_1,h_1-b_1\rangle+\alpha\langle h_2,h_2-b_2\rangle=|h_1|^2-\langle h_1,b_1\rangle+\alpha|h_2|^2-\alpha\langle h_2,b_2\rangle
		\end{equation*}
		and on the other hand, using the definitions \eqref{def h 1 e h 2}, \eqref{def di b1 e b 2} and \eqref{def di Y per thm derd},
		\begin{equation*}
		\langle h_1,h_1-b_1\rangle+\alpha\langle h_2,h_2-b_2\rangle=\left\langle \mbox{Ric}^{\varphi}+\mbox{Hess}(f),\frac{1}{2}\mathcal{L}_{Y}g\right\rangle +\alpha\langle \tau(\varphi)-d\varphi(\nabla f),d\varphi(Y)\rangle-\lambda\mbox{div}(Y),
		\end{equation*}
		then combining the above relations with \eqref{id per mostr form derd} we get
		\begin{equation*}
		\int_M[|h_1|^2-\langle h_1,b_1\rangle+\alpha|h_2|^2-\alpha\langle h_2,b_2\rangle]e^{-f}=\int_M\left(\frac{1}{2}S^{\varphi}+\Delta f-\frac{1}{2}|\nabla f|^2\right)\mbox{div}(e^{-f}Y)\mu-\lambda\int_M\mbox{div}(Y)e^{-f}.
		\end{equation*}
		The validity of \eqref{form derdinzky} follows from the above, since using that
		\begin{equation*}
		\mbox{div}(e^{-f}Y)=\mbox{div}(Y)e^{-f}-e^{-f}\langle \nabla f,Y\rangle, \quad \mbox{div}(fe^{-f}Y)=f\mbox{div}(e^{-f}Y)+e^{-f}\langle \nabla f,Y\rangle
		\end{equation*}
		we get
		\begin{equation*}
		\int_M\mbox{div}(Y)e^{-f}=\int_M\langle \nabla f,Y\rangle e^{-f}=-\int_Mf\mbox{div}(e^{-f}Y). \qedhere
		\end{equation*}
	\end{proof}

	\subsection{Rigidity of $\varphi$-Cotton flat compact harmonic-Ricci solitons}\label{sect rig caso compatto phi cotton}
	
	Using \hyperref[thm solitone ricci harm compatto non rigido grad shrink]{Theorem \ref*{thm solitone ricci harm compatto non rigido grad shrink}} we able to prove the next
	\begin{thm}\label{thm phi cotton flat compact ricci solitons are rigid}
		Let $(M,g)$ be compact harmonic-Ricci soliton of dimension $m\geq 3$. If $(M,g)$ is $\varphi$-Cotton flat then $(M,g)$ is rigid.
	\end{thm}
	\begin{proof}
		Assume by contradiction that $(M,g)$ is not rigid. From \hyperref[thm solitone ricci harm compatto non rigido grad shrink]{Theorem \ref*{thm solitone ricci harm compatto non rigido grad shrink}} we can assume that $(M,g)$ is a gradient (shrinking) harmonic-Ricci soliton, i.e,
		\begin{equation}\label{harmonic gradient ricci soliton euqation per dim caso compatto}
		\begin{dcases}
		\mbox{Ric}^{\varphi}+\mbox{Hess}(f)=\lambda g\\
		\tau(\varphi)=d\varphi(\nabla f)
		\end{dcases}
		\end{equation}
		holds for some $f\in\mathcal{C}^{\infty}(M)$, $\lambda>\mathbb{R}$, $\varphi:M\to N$ smooth, where $(N,\langle\,,\,\rangle_N)$ is a Riemannian manifold and $\alpha>0$.
		
		First of all, relying only on the valdity of \eqref{harmonic gradient ricci soliton euqation per dim caso compatto} (and without using that $\lambda>0$), we show that
		\begin{equation}\label{formula integrale nel caso generale}
		\int_M|\nabla \mbox{Ric}^{\varphi}|^2e^{-f}=\int_M|F^{\varphi}|^2e^{-f}+2\alpha\int_M\varphi^a_{kki}\varphi^a_jR^{\varphi}_{ij}e^{-f},
		\end{equation}
		where $F^{\varphi}$ is the tensor defined in \eqref{def di F phi}. Afterwards we will see how the validity of above implies rigidity assuming that $(M,g)$ is $\varphi$-Cotton flat.
		
		The following Weitzenb\"{o}ck identity holds
		\begin{equation*}
		\frac{1}{2}\Delta_f|\mbox{Ric}^{\varphi}|^2=|\nabla \mbox{Ric}^{\varphi}|^2+\langle \Delta_f\mbox{Ric}^{\varphi},\mbox{Ric}^{\varphi}\rangle.
		\end{equation*}
		Multiplying the above by $e^{-f}$, integrating it on $M$ and using the divergence theorem we get
		\begin{equation*}
		\int_M|\nabla \mbox{Ric}^{\varphi}|^2e^{-f}=-\int_M\langle \Delta_f\mbox{Ric}^{\varphi},\mbox{Ric}^{\varphi}\rangle e^{-f}.
		\end{equation*}
		Contracting \eqref{f laplacian phi ricci per solitone grad scritt con hess f} against $\mbox{Ric}^{\varphi}$ we get
		\begin{equation*}
		\begin{aligned}
		\int_M\langle \Delta_f\mbox{Ric}^{\varphi},\mbox{Ric}^{\varphi}\rangle e^{-f}=&\int_MR^{\varphi}_{ij}\Delta_f R^{\varphi}_{ij}e^{-f}\\
		=&2\int_MR_{tikj}f_{kt}R^{\varphi}_{ij}e^{-f}-2\alpha\int_M\varphi^a_i\varphi^a_kf_{kj}R^{\varphi}_{ij}e^{-f}-2\alpha \int_M\varphi^a_{ij}R^{\varphi}_{ij}\varphi^a_{kk}e^{-f},
		\end{aligned}
		\end{equation*}
		hence the above relation gives
		\begin{equation}\label{integrale di ric phi contro f laplaciano di ric phi}
		\int_M|\nabla \mbox{Ric}^{\varphi}|^2e^{-f}=-2\int_MR_{tikj}f_{kt}R^{\varphi}_{ij}e^{-f}+2\alpha\int_M\varphi^a_i\varphi^a_kf_{kj}R^{\varphi}_{ij}e^{-f}+2\alpha \int_M\varphi^a_{ij}R^{\varphi}_{ij}\varphi^a_{kk}e^{-f}.
		\end{equation}
		
		Using the divergence theorem
		\begin{equation*}
		-\int_MR_{tikj}f_{kt}R^{\varphi}_{ij}e^{-f}=\int_M(R_{tikj}R^{\varphi}_{ij}e^{-f})_tf_k=\int_M(R_{tikj}e^{-f})_tR^{\varphi}_{ij}f_k+\int_Mf_kR_{tikj}R^{\varphi}_{ij,t}e^{-f},
		\end{equation*}
		that gives
		\begin{equation}\label{thm caso compatto form integrale primo pezzo}
		-2\int_MR_{tikj}f_{kt}R^{\varphi}_{ij}e^{-f}=2\int_M(R_{tikj}e^{-f})_tR^{\varphi}_{ij}f_k+2\int_Mf_tR_{tikj}R^{\varphi}_{ij,k}e^{-f}.
		\end{equation}
		Plugging \eqref{norma F quadro caso compatto} and \eqref{div riem con esponeziale caso compatto} into \eqref{thm caso compatto form integrale primo pezzo} we get
		\begin{align*}
		-2\int_MR_{tikj}f_{kt}R^{\varphi}_{ij}e^{-f}=&2\alpha\int_M(\varphi^a_{ik}\varphi^a_j-\varphi^a_{ij}\varphi^a_k)R^{\varphi}_{ij}f_ke^{-f}+\int_M|F^{\varphi}|^2e^{-f}\\
		=&\int_M|F^{\varphi}|^2e^{-f}+2\alpha\int_M\varphi^a_{ik}\varphi^a_jR^{\varphi}_{ij}f_ke^{-f}-2\alpha\int_M\varphi^a_{ij}\varphi^a_{kk}R^{\varphi}_{ij}e^{-f}
		\end{align*}
		and thus, from \eqref{integrale di ric phi contro f laplaciano di ric phi} we obtain,
		\begin{align*}
		\int_M|\nabla \mbox{Ric}^{\varphi}|^2e^{-f}=\int_M|F^{\varphi}|^2e^{-f}+2\alpha\int_M\varphi^a_{ik}\varphi^a_jR^{\varphi}_{ij}f_ke^{-f}+2\alpha\int_M\varphi^a_i\varphi^a_kf_{kj}R^{\varphi}_{ij}e^{-f}.
		\end{align*}
		Taking the covariant derivative of the second equation of \eqref{harmonic gradient ricci soliton euqation per dim caso compatto} we get $\varphi^a_{kki}=\varphi^a_{ki}f_k+\varphi^a_kf_{ki}$, hence from the above we conclude that \eqref{formula integrale nel caso generale} holds.
		
		Now we assume $C^{\varphi}=0$. Then $\varphi$ is harmonic and thus \eqref{formula integrale nel caso generale} reduces to
		\begin{equation*}
		\int_M|\nabla \mbox{Ric}^{\varphi}|^2e^{-f}=\int_M|F^{\varphi}|^2e^{-f},
		\end{equation*}
		so that, using \eqref{norma F phi quadro con C phi zero},
		\begin{equation*}
		\int_M|\nabla \mbox{Ric}^{\varphi}|^2e^{-f}=\frac{1}{2(m-1)}\int_M|\nabla S^{\varphi}|^2e^{-f}.
		\end{equation*}
		
		On the other hand via Cauchy-Schwarz inequality
		\begin{equation*}\label{dis per nabla ric phi}
		|\nabla \mbox{Ric}^{\varphi}|^2\geq \frac{1}{m}|\nabla S^{\varphi}|^2,
		\end{equation*}
		hence from the above we conclude that
		\begin{equation*}
		\frac{m-2}{2m(m-1)}\int_M|\nabla S^{\varphi}|^2e^{-f}\leq 0.
		\end{equation*}
		Then, since $m\geq 3$, we conclude that $S^{\varphi}$ is constant on $M$. Then, in view of \hyperref[rmk rigido in caso compatto se e solo se phi curv scalar cost]{Remark \ref*{rmk rigido in caso compatto se e solo se phi curv scalar cost}}, $(M,g)$ is rigid. Contradiction, hence the proof is concluded.
	\end{proof}
	
	\section{Complete non-compact gradient shrinking harmonic-Ricci solitons}\label{section noncompac shirnking}
	
	In this section $(M,g)$ is a complete non-compact harmonic-Ricci soliton, i.e., there exist $f\in\mathcal{C}^{\infty}(M)$, $\alpha,\lambda>0$, $\varphi:M\to N$ smooth, where $(N,\langle\,,\,\rangle_N)$ is a Riemannian manifold, such that
	\begin{equation}\label{sol equat non compact conti con cutoff}
	\begin{dcases}
	\mbox{Ric}^{\varphi}+\mbox{Hess}(f)=\lambda g\\
	\tau(\varphi)=d\varphi(\nabla f).
	\end{dcases}
	\end{equation}
	
	\begin{rmk}\label{rmk cutoff}
		Recall that on any complete Riemannian manifold there exists, for every $R>0$ and $p\in M$, a cutoff function $\rho_R$ such that $0\leq \rho_R\leq 1$, $\rho_R=1$ on $B_p(R)$, $\rho_R=0$ on $M\setminus B_p(2R)$ and $|\nabla \rho_R|\leq \frac{C}{R}$, where $C$ is a positive constant independent on $R>0$ and $p\in M$.
	\end{rmk}
	
	The following technical lemma shall be useful later on in the proof of \hyperref[thm caso non compatto con phi cotton nullo]{Theorem \ref*{thm caso non compatto con phi cotton nullo}}.
	\begin{lemma}
		Let $(M,g)$ be a complete non-compact gradient harmonic-Ricci soliton, that is, \eqref{sol equat non compact conti con cutoff} holds for some $f\in\mathcal{C}^{\infty}(M)$, $\alpha,\lambda>0$, $\varphi:M\to N$ smooth, where $(N,\langle\,,\,\rangle_N)$ is a Riemannian manifold. For every positive constant $\mu$ we have
		\begin{equation}\label{ric phi in l 2 e anche tensione}
		\int_M|\mbox{Ric}^{\varphi}|^2e^{-\mu f},\quad \int_M|\tau(\varphi)|^2e^{-\mu f}<+\infty.
		\end{equation}
		In particular, for $\mu=1$ we have the following relation
		\begin{equation}\label{eq che implica ric phi in l 2}
		\int_M|\mbox{Ric}^{\varphi}|^2e^{-f}+\alpha\int_M|\tau(\varphi)|^2e^{-f}= \lambda\int_MS^{\varphi}e^{-f}<+\infty.
		\end{equation}
		Furthermore
		\begin{equation}\label{grad di phi curv scalare sta in l 2}
			\int_M|\nabla S^{\varphi}|^2e^{-f}<+\infty.
		\end{equation}
	\end{lemma}
	\begin{proof}
		Using the first equation of \eqref{sol equat non compact conti con cutoff} and the divergence theorem, for a smooth cutoff $\rho$ we get
		\begin{equation*}
		\int_M|\mbox{Ric}^{\varphi}|^2\rho^2e^{-\mu f}=\int_MR^{\varphi}_{ij}(\lambda\delta_{ij}-f_{ij})\rho^2e^{-\mu f}=\lambda\int_MS^{\varphi}\rho^2e^{-f}+\int_M(R^{\varphi}_{ij}\rho^2e^{-\mu f})_jf_i.
		\end{equation*}
		Using also \eqref{div of phi Ricci}, the second equation of \eqref{sol equat non compact conti con cutoff} and \eqref{norma grad S phi in termini di ric phi quadro} the above gives
		\begin{equation}\label{prima stima per norma 2 di ric phi}
		\begin{aligned}
		\int_M|\mbox{Ric}^{\varphi}|^2\rho^2e^{-\mu f}+\alpha\int_M|\tau(\varphi)|^2\rho^2e^{-\mu f}=&\lambda\int_MS^{\varphi}\rho^2e^{-\mu f}+(1-\mu)\int_MR^{\varphi}_{ij}f_if_j\rho^2 e^{-\mu f}\\
		&+\int_MR^{\varphi}_{ij}f_i(\rho^2)_je^{-\mu f}.
		\end{aligned}
		\end{equation}
		
		Notice that
		\begin{equation}\label{norma di ric phi applicato a grad f}
		|\mbox{ric}^{\varphi}(\nabla f)|^2=R^{\varphi}_{ij}f_jR^{\varphi}_{ik}f_k=(\mbox{Ric}^{\varphi})^2(\nabla f,\nabla f),
		\end{equation}
		where we are denoting by $\mbox{ric}^{\varphi}$ the $(1,1)$-version of $\mbox{Ric}^{\varphi}$. Clearly $(\mbox{Ric}^{\varphi})^2\geq 0$ and thus the following inequality in the sense of quadratic forms holds
		\begin{equation*}
		(\mbox{Ric}^{\varphi})^2\leq \mbox{tr}[(\mbox{Ric}^{\varphi})^2]g=|\mbox{Ric}^{\varphi}|^2g,
		\end{equation*}
		hence
		\begin{equation}\label{stima ric phi quadro applicato a grad f}
		(\mbox{Ric}^{\varphi})^2(\nabla f,\nabla f)\leq |\mbox{Ric}^{\varphi}|^2|\nabla f|^2.
		\end{equation}
		Then \eqref{norma di ric phi applicato a grad f} yields the validity of
		\begin{equation}\label{norma di ric phi applicato a grad f stimata}
		|\mbox{ric}^{\varphi}(\nabla f)|\leq |\mbox{Ric}^{\varphi}||\nabla f|
		\end{equation}
		
		Now, from Cauchy-Schwarz inequality and \eqref{norma di ric phi applicato a grad f stimata} we deduce
		\begin{equation*}
		|R^{\varphi}_{ij}f_i(\rho^2)_j|\leq |\mbox{ric}^{\varphi}(\nabla f)||\nabla \rho^2|\leq |\mbox{Ric}^{\varphi}||\nabla f||\nabla \rho^2|,
		\end{equation*}
		so that, since $\nabla \rho=2\rho \nabla \rho$, using the following Cauchy inequality, valid for every $\varepsilon>0$ and for every $a,b\in\mathbb{R}$,
		\begin{equation}\label{cauhcy con epsilon}
		ab\leq \frac{\varepsilon}{2}a^2+\frac{1}{2\varepsilon}b^2,
		\end{equation}
		we get from the above
		\begin{align*}
		|R^{\varphi}_{ij}f_i(\rho^2)_j|\leq2\rho|\mbox{Ric}^{\varphi}||\nabla f||\nabla \rho|\leq 2\varepsilon|\mbox{Ric}^{\varphi}|^2\rho^2+\frac{1}{2\varepsilon}|\nabla f|^2|\nabla \rho|^2,
		\end{align*}
		that, for $\varepsilon=\frac{1}{8}$ reads
		\begin{equation}\label{stima che mi serve per far vedere ric phi in l 2}
		|R^{\varphi}_{ij}f_i(\rho^2)_j|\leq \frac{1}{4}|\mbox{Ric}^{\varphi}|^2\rho^2+4|\nabla f|^2|\nabla \rho|^2.
		\end{equation}
		
		Furthermore, using Cauchy-Schwarz inequality and \eqref{cauhcy con epsilon} we have
		\begin{equation*}
			|R^{\varphi}_{ij}f_if_j|\leq |\mbox{Ric}^{\varphi}||\nabla f|^2\leq \frac{\varepsilon}{2}|\mbox{Ric}^{\varphi}|^2+\frac{1}{2\varepsilon}|\nabla f|^4,
		\end{equation*}
		that for $\varepsilon=\frac{1}{2|1-\mu|}$ (when $\mu\neq 1$, because when $\mu=1$ we do not have to deal with this term) gives
		\begin{equation}\label{stima pezzo con mu per vedere ric phi in l 2}
			(1-\mu)\int_MR^{\varphi}_{ij}f_if_j\rho^2 e^{-\mu f}\leq \frac{1}{4}\int_M|\mbox{Ric}^{\varphi}|^2\rho^2 e^{-\mu f}+(1-\mu)^2\int_M|\nabla f|^4\rho^2 e^{-\mu f}
		\end{equation}
		Then, with the aid of \eqref{stima che mi serve per far vedere ric phi in l 2} and \eqref{stima pezzo con mu per vedere ric phi in l 2}, from \eqref{prima stima per norma 2 di ric phi} we get
		\begin{equation}\label{eq per mostrare che ric phi sta in l 2}
		\begin{aligned}
		\frac{1}{2}\int_M|\mbox{Ric}^{\varphi}|^2\rho^2e^{-\mu f}+\alpha\int_M|\tau(\varphi)|^2\rho^2e^{-\mu f}\leq& \lambda\int_MS^{\varphi}\rho^2e^{-\mu f}+(1-\mu)^2\int_M|\nabla f|^4\rho^2 e^{-\mu f}\\
		&+2\int_M|\nabla f|^2|\nabla \rho|^2e^{-\mu f}.
		\end{aligned}
		\end{equation}
		
		Recalling the validity of \eqref{finitezza inte di s phi e nabla f quadro}, from \eqref{eq per mostrare che ric phi sta in l 2} and choosing $\rho=\rho_R$ as defined in \hyperref[rmk cutoff]{Remark \ref*{rmk cutoff}}, we get
		\begin{align*}
		\frac{1}{2}\int_{B_p(R)}|\mbox{Ric}^{\varphi}|^2e^{-\mu f}+\alpha\int_{B_p(R)}|\tau(\varphi)|^2e^{-\mu f}\leq& \lambda\int_MS^{\varphi}e^{-\mu f}+(1-\mu)^2\int_M|\nabla f|^4e^{-\mu f}\\
		&+\frac{1}{R}\int_M|\nabla f|^2e^{-\mu f}<+\infty,
		\end{align*}
		Letting $R\to +\infty$ we get \eqref{ric phi in l 2 e anche tensione}, since $\alpha>0$.
		
		Once we know that $\mbox{Ric}^{\varphi}\in L^2(M,e^{-f})$, using \eqref{prima stima per norma 2 di ric phi} with $\rho=\rho_R$ and $\mu=1$ and passing to the limit for $R\to +\infty$ we obtain \eqref{eq che implica ric phi in l 2}. Indeed it is sufficient to show that
		\begin{equation}
		\lim_{R\to +\infty}\int_MR^{\varphi}_{ij}f_i(\rho^2_R)_je^{-f}=0,
		\end{equation}
		and this follows easily from the inequality
		\begin{equation*}
			|R^{\varphi}_{ij}f_i(\rho^2)_j|\leq 2\rho|\mbox{Ric}^{\varphi}||\nabla f||\nabla \rho|\leq \frac{2C}{R}|\mbox{Ric}^{\varphi}||\nabla f|\leq \frac{C}{R}(|\mbox{Ric}^{\varphi}|^2+|\nabla f|^2)
		\end{equation*}
		and \eqref{finitezza inte di s phi e nabla f quadro}.
		
		It remains to prove \eqref{grad di phi curv scalare sta in l 2}. Using \eqref{form gradiente phi curv scalare in solitoni gradiente} twice we have
		\begin{equation}\label{norma grad S phi in termini di ric phi quadro}
		|\nabla S^{\varphi}|^2=4(\mbox{Ric}^{\varphi})^2(\nabla f,\nabla f),
		\end{equation}
		that with the aid of \eqref{stima ric phi quadro applicato a grad f} gives
		\begin{equation*}
			|\nabla S^{\varphi}|^2e^{-f}\leq |\mbox{Ric}^{\varphi}|^2|\nabla f|^2e^{-f}.
		\end{equation*}
		Since $f$ has polynomial growth we have
		\begin{equation*}
			|\nabla f|^2e^{-f}\leq e^{-\mu f}
		\end{equation*}
		for some $0< \mu<1$. Then we deduce, using \eqref{norma grad S phi in termini di ric phi quadro} and the above,
		\begin{equation*}
			\int_M|\nabla S^{\varphi}|^2e^{-f}\leq 4\int_M|\mbox{Ric}^{\varphi}|^2e^{-\mu f},
		\end{equation*}
		that is finite in view of \eqref{ric phi in l 2 e anche tensione}.
	\end{proof}
	
	Now we are ready to prove the key result in order to extend the validity of \hyperref[thm phi cotton flat compact ricci solitons are rigid]{Theorem \ref*{thm phi cotton flat compact ricci solitons are rigid}} to complete non-compact gradient solitons.
	\begin{thm}\label{thm caso non compatto con phi cotton nullo}
		Let $(M,g)$ be a complete non-compact gradient harmonic-Ricci soliton of dimension $m\geq 2$, that is, \eqref{sol equat non compact conti con cutoff} holds for some $f\in\mathcal{C}^{\infty}(M)$, $\alpha,\lambda>0$, $\varphi:M\to N$ smooth, where $(N,\langle\,,\,\rangle_N)$ is a Riemannian manifold. If $C^{\varphi}=0$ then
		\begin{equation}\label{eq fond che lega grad di s phi con norma di der cov di phi ricci}
			\int_M|\nabla \mbox{Ric}^{\varphi}|^2e^{-f}=\frac{1}{2(m-1)}\int_M|\nabla S^{\varphi}|^2e^{-f}<+\infty.
		\end{equation}
		In particular, if $m\geq 3$, $\nabla \mbox{Ric}^{\varphi}=0$.
	\end{thm}
	\begin{proof}
		Let $\rho$ be a smooth function with compact support, then, in a local orthonormal coframe
		\begin{equation*}
		|\nabla \mbox{Ric}^{\varphi}|^2\rho^2e^{-f}=(R^{\varphi}_{ij}R^{\varphi}_{ij,k}\rho^2e^{-f})_k-R^{\varphi}_{ij}(R^{\varphi}_{ij,k}\rho^2e^{-f})_k.
		\end{equation*}
		Integrating the above and using the divergence theorem we get
		\begin{equation}\label{integrale div ric phi caso non compatto}
		\int_M|\nabla \mbox{Ric}^{\varphi}|^2\rho^2e^{-f}=-\int_MR^{\varphi}_{ij}(R^{\varphi}_{ij,k}\rho^2e^{-f})_k=-\int_MR^{\varphi}_{ij}\Delta_fR^{\varphi}_{ij}\rho^2e^{-f}-\int_MR^{\varphi}_{ij}R^{\varphi}_{ij,k}(\rho^2)_ke^{-f}.
		\end{equation}
		
		Contracting \eqref{f laplacian phi ricci per solitone grad scritt con hess f} against $\mbox{Ric}^{\varphi}$ we get
		\begin{equation*}
		R^{\varphi}_{ij}\Delta_f R^{\varphi}_{ij}=2R_{tikj}f_{kt}R^{\varphi}_{ij}-2\alpha\varphi^a_i\varphi^a_kf_{kj}R^{\varphi}_{ij}-2\alpha \varphi^a_{ij}R^{\varphi}_{ij}\varphi^a_{kk}.
		\end{equation*}
		Since $C^{\varphi}=0$ we know that $\varphi$ is harmonic, hence the above becomes
		\begin{equation*}
		R^{\varphi}_{ij}\Delta_f R^{\varphi}_{ij}=2R_{tikj}f_{kt}R^{\varphi}_{ij}-2\alpha\varphi^a_i\varphi^a_kf_{kj}R^{\varphi}_{ij},
		\end{equation*}
		and thus \eqref{integrale div ric phi caso non compatto} can be rewritten as
		\begin{equation}\label{caso non comp integrale di ric phi contro f laplaciano di ric phi}
		\int_M|\nabla \mbox{Ric}^{\varphi}|^2\rho^2e^{-f}=-2\int_MR_{tikj}f_{kt}R^{\varphi}_{ij}\rho^2 e^{-f}+2\alpha\int_M\varphi^a_i\varphi^a_kf_{kj}R^{\varphi}_{ij}\rho^2e^{-f}-\int_MR^{\varphi}_{ij}R^{\varphi}_{ij,k}(\rho^2)_ke^{-f}.
		\end{equation}
		Integrating by parts
		\begin{align*}
		-\int_MR_{tikj}f_{kt}R^{\varphi}_{ij}\rho^2e^{-f}=&\int_M(R_{tikj}R^{\varphi}_{ij}\rho^2e^{-f})_tf_k\\
		=&\int_M(R_{tikj}e^{-f})_tR^{\varphi}_{ij}f_k\rho^2+\int_Mf_kR_{tikj}R^{\varphi}_{ij,t}\rho^2e^{-f}+\int_MR_{tikj}f_kR^{\varphi}_{ij}(\rho^2)_te^{-f}.
		\end{align*}
		By plugging \eqref{norma F quadro caso compatto} and \eqref{div riem con esponeziale caso compatto} into the above we get
		\begin{align*}
		-\int_MR_{tikj}f_{kt}R^{\varphi}_{ij}\rho^2e^{-f}=&\alpha\int_M(\varphi^a_{ik}\varphi^a_j-\varphi^a_{ij}\varphi^a_k)R^{\varphi}_{ij}f_k\rho^2e^{-f}+\frac{1}{2}\int_M|F^{\varphi}|^2\rho^2e^{-f}+\int_MR_{tikj}f_kR^{\varphi}_{ij}(\rho^2)_te^{-f},
		\end{align*}
		that gives, using $\varphi^a_kf_k=0$ and $\varphi^a_{ik}f_k=(\varphi^a_kf_k)_i-\varphi^a_kf_{ki}=-\varphi^a_kf_{ki}$,
		\begin{align*}
		-\int_MR_{tikj}f_{kt}R^{\varphi}_{ij}\rho^2e^{-f}=\frac{1}{2}\int_M|F^{\varphi}|^2\rho^2e^{-f}-\alpha\int_M\varphi^a_k\varphi^a_jR^{\varphi}_{ij}f_{ki}\rho^2e^{-f}+\int_MR_{tikj}f_kR^{\varphi}_{ij}(\rho^2)_te^{-f}.
		\end{align*}
		
		Inserting the above into \eqref{caso non comp integrale di ric phi contro f laplaciano di ric phi} we conclude
		\begin{equation}\label{stima utile per norma di der cov di phi ricci}
		\int_M|\nabla \mbox{Ric}^{\varphi}|^2\rho^2e^{-f}=\int_M|F^{\varphi}|^2\rho^2e^{-f}+2\int_MR_{tikj}f_kR^{\varphi}_{ij}(\rho^2)_te^{-f}-\int_MR^{\varphi}_{ij}R^{\varphi}_{ij,k}(\rho^2)_ke^{-f}.
		\end{equation}
		
		Since $C^{\varphi}=0$ we know that \eqref{norma F phi quadro con C phi zero} holds, hence
		\begin{equation*}
		|F^{\varphi}|^2=\frac{1}{2(m-1)}|\nabla S^{\varphi}|^2.
		\end{equation*}
		In view of \eqref{grad di phi curv scalare sta in l 2} we know that $F^{\varphi}\in L^2(M,e^{-f})$.
		
		To prove \eqref{eq fond che lega grad di s phi con norma di der cov di phi ricci} it is sufficient to show the validity of
		\begin{equation}\label{primo limite fa zero}
		\lim_{R\to +\infty}\int_MR_{tikj}f_kR^{\varphi}_{ij}(\rho_R^2)_te^{-f}= 0
		\end{equation}
		and
		\begin{equation}\label{secondo limite fa zero}
		\lim_{R\to +\infty}\int_MR^{\varphi}_{ij}R^{\varphi}_{ij,k}(\rho_R^2)_ke^{-f}=0,
		\end{equation}
		where $\rho_R$ is the cutoff function defined in \hyperref[rmk cutoff]{Remark \ref*{rmk cutoff}}. Indeed, assuming the validity of \eqref{primo limite fa zero} and \eqref{secondo limite fa zero}, passing to the limit for $R\to +\infty$ in \eqref{stima utile per norma di der cov di phi ricci} with $\rho=\rho_R$ we obtain \eqref{eq fond che lega grad di s phi con norma di der cov di phi ricci}.
		
		We start proving \eqref{primo limite fa zero}. Using \eqref{form commutazione der covariante per ricci phi in solitoni gradiente} we have
		\begin{equation*}
		R_{tikj}f_kR^{\varphi}_{ij}=R_{kitj}f_kR^{\varphi}_{ij}=F^{\varphi}_{itj}R^{\varphi}_{ij}
		\end{equation*}
		and thus, using Cauchy-Schwarz inequality,
		\begin{equation*}
		|R_{tikj}f_kR^{\varphi}_{ij}(\rho_R^2)_t|=|F^{\varphi}_{itj}R^{\varphi}_{ij}(\rho^2_R)_t|\leq |F^{\varphi}||\mbox{Ric}^{\varphi}||\nabla\rho_R^2|.
		\end{equation*}
		Then, using the property of the cutoff $\rho_R$, on $B_p(2R)\setminus B_p(R)$
		\begin{equation*}
		|R_{tikj}f_kR^{\varphi}_{ij}(\rho_R^2)_t|\leq 2\rho_R|F^{\varphi}||\mbox{Ric}^{\varphi}||\nabla\rho_R|\leq \frac{2C}{R}\left(\frac{1}{2}|F^{\varphi}|^2+\frac{1}{2}|\mbox{Ric}^{\varphi}|^2\right)=\frac{C}{R}\left(|F^{\varphi}|^2+|\mbox{Ric}^{\varphi}|^2\right),
		\end{equation*}
		so that, using that $\mbox{Ric}^{\varphi},F^{\varphi}\in L^2(M,e^{-f})$, we guarantee the existence of a positive constant $C_1$ independent from $R$ such that
		\begin{equation*}
		\left|\int_MR_{tikj}f_kR^{\varphi}_{ij}(\rho_R^2)_te^{-f}\right|\leq \int_M|R_{tikj}f_kR^{\varphi}_{ij}(\rho_R^2)_t|e^{-f}\leq \frac{C}{R}\left(\int_M|F^{\varphi}|^2e^{-f}+\int_M|\mbox{Ric}^{\varphi}|^2e^{-f}\right)=\frac{C_1}{R},
		\end{equation*}
		and thus \eqref{primo limite fa zero} follows.
		
		Now is the turn of \eqref{secondo limite fa zero}. To prove if first of all we need to show that
		\begin{equation}\label{finitezza norma 2 di der cov di phi ricci}
		\int_M|\nabla \mbox{Ric}^{\varphi}|^2e^{-f}<+\infty.
		\end{equation}
		Using Cauchy-Schwarz inequality and the Cauchy inequality \eqref{cauhcy con epsilon} we have, for every $\varepsilon>0$,
		\begin{equation*}
		|R^{\varphi}_{ij}R^{\varphi}_{ij,k}(\rho^2_R)_k|\leq |\nabla \mbox{Ric}^{\varphi}||\mbox{Ric}^{\varphi}||\nabla \rho^2_R|=2\rho_R|\nabla \mbox{Ric}^{\varphi}||\mbox{Ric}^{\varphi}|\leq \left(2\varepsilon\rho^2_R|\nabla \mbox{Ric}^{\varphi}|^2+\frac{1}{2\varepsilon}|\mbox{Ric}^{\varphi}|^2\right),
		\end{equation*}
		so that, for $\varepsilon=\frac{1}{4}$ from \eqref{stima utile per norma di der cov di phi ricci} with $\rho=\rho_R$ we deduce
		\begin{align*}
		\int_M|\nabla \mbox{Ric}^{\varphi}|^2\rho_R^2e^{-f}=&\int_M|F^{\varphi}|^2\rho_R^2e^{-f}+2\int_MR_{tikj}f_kR^{\varphi}_{ij}(\rho^2_R)_te^{-f}\\
		&+\frac{1}{2}\int_M\rho^2_R|\nabla \mbox{Ric}^{\varphi}|^2e^{-f}+2\int_M|\mbox{Ric}^{\varphi}|^2\rho_R^2e^{-f},
		\end{align*}
		and using that $\mbox{Ric}^{\varphi},F^{\varphi}\in L^2(M,e^{-f})$ the above gives
		\begin{align*}
		\frac{1}{2}\int_M|\nabla \mbox{Ric}^{\varphi}|^2\rho_R^2e^{-f}=2\int_MR_{tikj}f_kR^{\varphi}_{ij}(\rho_R^2)_te^{-f}+C_1.
		\end{align*}
		With the aid of \eqref{primo limite fa zero} we conclude that \eqref{finitezza norma 2 di der cov di phi ricci} holds.
		
		Now that we obtained the validity of \eqref{finitezza norma 2 di der cov di phi ricci} we are finally ready to prove \eqref{secondo limite fa zero}. Using Cauchy-Schwarz inequality and the properties of the cutoff $\rho_R$,
		\begin{equation*}
		|R^{\varphi}_{ij}R^{\varphi}_{ij,k}(\rho^2_R)_k|\leq |\nabla \mbox{Ric}^{\varphi}||\mbox{Ric}^{\varphi}||\nabla \rho^2_R|\leq\frac{C}{R}(|\nabla \mbox{Ric}^{\varphi}|^2+|\mbox{Ric}^{\varphi}|^2),
		\end{equation*}
		hence, for some positive constant independent from $R$,
		\begin{equation*}
		\int_MR^{\varphi}_{ij}R^{\varphi}_{ij,k}(\rho^2_R)_ke^{-f}\leq \int_M|R^{\varphi}_{ij}R^{\varphi}_{ij,k}(\rho^2_R)_k|e^{-f}\leq \frac{C}{R}\left(\int_M|\nabla \mbox{Ric}^{\varphi}|^2e^{-f}+\int_M|\mbox{Ric}^{\varphi}|^2e^{-f}\right)=\frac{C_1}{R}.
		\end{equation*}
		Now \eqref{secondo limite fa zero} follows and thus \eqref{eq fond che lega grad di s phi con norma di der cov di phi ricci} holds.
		
		Using the inequality \eqref{dis per nabla ric phi}, the validity of \eqref{eq fond che lega grad di s phi con norma di der cov di phi ricci} immediately gives that $S^{\varphi}$ is constant on $M$ and thus $\mbox{Ric}^{\varphi}$ is parallel when $m=3$, as seen in the proof of \hyperref[thm phi cotton flat compact ricci solitons are rigid]{Theorem \ref*{thm phi cotton flat compact ricci solitons are rigid}}.
	\end{proof}
	
	The above Theorem, combined with \hyperref[thm rigidity]{Theorem \ref*{thm rigidity}} and \hyperref[rmk thm rigidity per shrinking]{Remark \ref*{rmk thm rigidity per shrinking}} gives the following
	\begin{cor}\label{cor se c phi zero allora rigidita}
		Let $(M,g)$ be a complete non-compact harmonic-Ricci soliton of dimension $m\geq 3$, that is, \eqref{sol equat non compact conti con cutoff} holds for some $f\in\mathcal{C}^{\infty}(M)$, $\alpha,\lambda>0$, $\varphi:M\to N$ smooth, where $(N,\langle\,,\,\rangle_N)$ is a Riemannian manifold. If $C^{\varphi}=0$ then $(M,g)$ is isometric to a finite quotient of $L\times \mathbb{R}^k$ for some $1\leq k\leq m$, where $(L,g_L)$ is a compact harmonic-Einstein manifold (with respect to $\alpha$, $\varphi_L:M\to N$ and $\lambda$) and, via the isometry, $\varphi=\varphi_L\circ \pi_L$ and $f=f_{\mathbb{R}^k}\circ \pi_{\mathbb{R}^k}$, where $f_{\mathbb{R}^k}=\frac{\lambda}{2}|x|^2+\langle b,x\rangle+c$ for some $c\in\mathbb{R}$ and $b\in\mathbb{R}^k$ and $\pi_L:L\times \mathbb{R}^k\to L$ and $\pi_{\mathbb{R}^k}:L\times \mathbb{R}^k\to \mathbb{R}^k$ are the canonical projections.
	\end{cor}
	
	\section{Rigidity with assumptions on the $\varphi$-Bach tensor}\label{section bach flat}
	
	In this Section our aim is to extend the results of \cite{CC} to the class of harmonic-Ricci solitons, where it is natural to replace the Bach tensor with the $\varphi$-Bach tensor. We begin by showing in the next Theorem that we can reduce ourselves to the classification of the previous Sections.
	\begin{thm}\label{thm che bach flat implica phi cotton flat}
		Let $(M,g)$ be a complete gradient shrinking harmonic-Ricci soliton of dimension $m\geq 3$. If the totally traceless part $\mathring{B}^{\varphi}$ of the $\varphi$-Bach tensor vanishes on $M$, then $(M,g)$ is $\varphi$-Cotton flat.
	\end{thm}
	\begin{proof}
		The theorem above has been essentially proved in Chapter 6 of \cite{A} but, since here we need a little modification in one of the assumptions, we briefly recall how the proof works. Chapter 6 of \cite{A} deals with complete gradient Einstein type structures of dimension $m\geq 3$, i.e., complete Riemannian manifolds $(M,g)$ such that
		\begin{equation}\label{einstein type structure gradient}
		\begin{cases}
		\mbox{Ric}^{\varphi}+\mbox{Hess}(f)-\mu df\otimes df=\lambda g\\
		\tau(\varphi)=d\varphi(\nabla f),
		\end{cases}
		\end{equation}
		for some $\alpha\in\mathbb{R}\setminus\{0\}$, $\mu\in\mathbb{R}$, $\lambda,f\in\mathcal{C}^{\infty}(M)$ and $\varphi:M\to N$, where $(N,\langle\,,\,\rangle_N)$ is a Riemannian manifold. We are interested to the more particular situation where $\mu=0$, $\alpha>0$ and $\lambda\in\mathbb{R}$, that is, the situation where $(M,g)$ is a complete gradient harmonic-Ricci soliton of dimension $m\geq 3$.
		
		Notice that, if $f$ is constant then $(M,g)$ is harmonic-Einstein and thus $C^{\varphi}=0$ is trivially satisfied. Otherwise, if $f$ is non-constant, then it is proper. Indeed, in the compact case the statement is trivial, while in the complete non-compact case we rely on the estimates for the potential function, see \hyperref[rmk su properness potential funct]{Remark \ref*{rmk su properness potential funct}}, to obtain its properness.
		
		In Proposition 6.1.10 of \cite{A} we proved
		\begin{equation}\label{first integrability condition with general mu}
		C^{\varphi}_{ijk}+f_tW^{\varphi}_{tijk}=D^{\varphi}_{ijk},
		\end{equation}
		where, in a local orthonormal coframe, the components of $D^{\varphi}$ are given by
		\begin{equation}\label{definition of D phi}
		D^{\varphi}_{ijk}:=\frac{1}{m-2}\left[R^{\varphi}_{ij}f_k-R^{\varphi}_{ik}f_j+\frac{1}{m-1}f_t(R^{\varphi}_{tk}\delta_{ij}-R^{\varphi}_{tj}\delta_{ik})-\frac{S^{\varphi}}{m-1}(f_k\delta_{ij}-f_j\delta_{ik})\right].
		\end{equation}
		The tensor $D^{\varphi}$ has the following geometric meaning, as pointed out in Remark 6.1.4 of \cite{A}
		\begin{equation*}
		\widetilde{C}^{\varphi}=D^{\varphi},
		\end{equation*}
		where $\widetilde{C}^{\varphi}$ is the $\varphi$-Cotton tensor with respect to the conformal metric $\widetilde{g}:=e^{-\frac{2}{m-2}f}g$.
		
		In Proposition 6.1.15 of \cite{A} we proved
		\begin{equation}\label{second integrability condition with general mu}
		\begin{aligned}
		(m-2)B^{\varphi}_{ij}-\frac{m-3}{m-2}C^{\varphi}_{jik}f_k=D^{\varphi}_{ijk,k}-\frac{\alpha}{m-2}\varphi^a_{kk}\varphi^a_if_j.
		\end{aligned}
		\end{equation}
		
		In Proposition 6.2.3 of \cite{A}, assuming
		\begin{equation}\label{cond su bach della tesi}
		B^{\varphi}(\nabla f,\cdot)=0
		\end{equation}
		we proved the validity of 
		\begin{equation}\label{equation relating the divergence of Y and the norm of D phi and the tension of phi}
		\frac{m-2}{2}|D^{\varphi}|^2+\frac{\alpha}{m-2}|\tau(\varphi)|^2|\nabla f|^2=\mbox{div}(Y),
		\end{equation}
		where the components of the vector field $Y$ are given by, in a local orthonormal coframe,
		\begin{equation}\label{definition of Y}
		Y^k:=D^{\varphi}_{ijk}f_if_j.
		\end{equation}
		
		Formula \eqref{equation relating the divergence of Y and the norm of D phi and the tension of phi} is the key point to obtain Theorem 2.9 of \cite{A}, that guarantees that under the assumption \eqref{cond su bach della tesi}, if the potential is proper and non-constant then $D^{\varphi}=0$ and $\tau(\varphi)=0$. Indeed to prove the theorem one integrates \eqref{equation relating the divergence of Y and the norm of D phi and the tension of phi} on the sublevel of the potential function to deduce the validity of
		\begin{equation*}
		\frac{m-2}{2}\int_{M}|D^{\varphi}|^2+\frac{\alpha}{m-2}\int_{M}|\tau(\varphi)|^2|\nabla f|^2=0.
		\end{equation*}
		Finally, by studying the geometry of the level sets of the potential function in Proposition 6.3.29 we obtain that, if $D^{\varphi}=0$ and $\tau(\varphi)=0$, then $C^{\varphi}=0$ on $\{\nabla f\neq 0\}$. Since we are assuming $\lambda\in\mathbb{R}$ the potential function is real analytic in harmonic coordinates, see Remark 6.3.31 in \cite{A}. Then, from Remark 6.3.30, we have $C^{\varphi}=0$ on the whole $M$.
		
		We show that instead of \eqref{cond su bach della tesi} one could assume
		\begin{equation}\label{nuova assunzione su bach}
		B^{\varphi}\geq -\alpha\frac{1-\varepsilon}{(m-2)^2}|\tau(\varphi)|^2g,
		\end{equation}
		for some $\varepsilon>0$ to obtain that $D^{\varphi}$ and $\tau(\varphi)$ vanishes and then the vanishing of $C^{\varphi}$ on the whole $M$. Indeed, from the second integrability condition \eqref{second integrability condition with general mu} we easily get
		\begin{equation*}
		(m-2)B^{\varphi}(\nabla f,\nabla f)=D^{\varphi}_{ijk,k}f_if_j-\frac{\alpha}{m-2}|\tau(\varphi)|^2|\nabla f|^2,
		\end{equation*}
		and using \eqref{nuova assunzione su bach} we obtain
		\begin{equation*}
		D^{\varphi}_{ijk,k}f_if_j-\alpha\varepsilon|\tau(\varphi)|^2|\nabla f|^2\geq 0.
		\end{equation*}
		The above, using the relation $(6.2.8)$ of \cite{A},
		\begin{equation}\label{norm of D phi in terms of phi Ricci}
		|D^{\varphi}|^2=\frac{2}{m-2}D^{\varphi}_{ijk}R^{\varphi}_{ij}f_k,
		\end{equation}
		and the definition \eqref{definition of Y} of $Y$ gives
		\begin{equation*}
		\frac{m-2}{2}|D^{\varphi}|^2+\alpha\varepsilon|\tau(\varphi)|^2|\nabla f|^2\leq \mbox{div}(Y^{\varphi}).
		\end{equation*}
		Integrating the above on the sublevel of the potential function we deduce
		\begin{equation*}
		\frac{m-2}{2}\int_{M}|D^{\varphi}|^2+\alpha\varepsilon\int_{M}|\tau(\varphi)|^2|\nabla f|^2\leq 0,
		\end{equation*}
		Then we conclude exactly as in the previous case.
		
		To conclude the proof it only remains to show that the assumption \eqref{nuova assunzione su bach} is satisfied. To see that observe that the vanishing of the totally traceless part of $\varphi$-Bach implies \eqref{nuova assunzione su bach}, for some $\varepsilon>0$. Notice that, using \eqref{traccia phi bach}, the vanishing of $\mathring{B}^{\varphi}$ is equivalent to
		\begin{equation}
		B^{\varphi}=\alpha\frac{m-4}{m(m-2)^2}|\tau(\varphi)|^2g,
		\end{equation}
		that gives \eqref{nuova assunzione su bach} with the equality sign when choosing $\varepsilon=\frac{2(m-2)}{m}>0$, since $m\geq 3$.
	\end{proof}
	\begin{rmk}\label{rmk totally traceless part}
		Clearly in the Theorem above one could assume $B^{\varphi}=0$ instead of $\mathring{B}^{\varphi}=0$. When $m=4$ the $\varphi$-Bach tensor is traceless, from \eqref{traccia phi bach}, hence the two assumptions are actually the same. The reason why we preferred the latter assumption in the statement of the Theorem above is that, in dimension $m\neq 4$, the vanishing of $\varphi$-Bach implies automatically the harmonicity of $\varphi$ while the vanishing of $\mathring{B}^{\varphi}=0$ does not. Hence the assumption $\mathring{B}^{\varphi}=0$ does not require a priori that $\varphi$ is harmonic, exactly as assumption \eqref{cond su bach della tesi} of Chapter 6 of \cite{A}, and it is a geometric assumption on $M$ that does not involve the potential function $f$.
	\end{rmk}	

	As a consequence of \hyperref[thm che bach flat implica phi cotton flat]{Theorem \ref*{thm che bach flat implica phi cotton flat}} we get the following two Corollaries.
	\begin{cor}\label{cor bach flat harm ricci soliton sono harm einst}
		Let $(M,g)$ be a compact harmonic-Ricci soliton of dimension $m\geq 3$. If the totally traceless part $\mathring{B}^{\varphi}$ of the $\varphi$-Bach tensor vanishes on $M$, then $(M,g)$ is rigid.
	\end{cor}
	\begin{proof}
		From \hyperref[thm che bach flat implica phi cotton flat]{Theorem \ref*{thm che bach flat implica phi cotton flat}} we obtain that $(M,g)$ is $\varphi$-Cotton flat. Then the thesis follows from \hyperref[thm phi cotton flat compact ricci solitons are rigid]{Theorem \ref*{thm phi cotton flat compact ricci solitons are rigid}}.
	\end{proof}
	\begin{cor}\label{cor bach non comp flat harm ricci soliton sono harm einst}
		Let $(M,g)$ be a complete non-compact gradient shrinking harmonic-Ricci soliton of dimension $m\geq 3$ with respect to $f\in\mathcal{C}^{\infty}(M)$, $\alpha,\lambda>0$, $\varphi:M\to N$ smooth, where $(N,\langle\,,\,\rangle_N)$ is a Riemannian manifold. If the totally traceless part $\mathring{B}^{\varphi}$ of the $\varphi$-Bach tensor vanishes on $M$, then $(M,g)$ is isometric to a finite quotient of $L\times \mathbb{R}$, where $(L,g_L)$ is a compact harmonic-Einstein manifold with respect to $\alpha$, $\varphi_L:M\to N$ and $\lambda$ and, via the isometry, $\varphi=\varphi_L\circ \pi_L$ and $f=f_{\mathbb{R}}\circ \pi_{\mathbb{R}}$, where $f_{\mathbb{R}}=\frac{\lambda}{2}x^2+bx+c$ for some $b,c\in\mathbb{R}$ and $\pi_L:L\times \mathbb{R}\to L$ and $\pi_{\mathbb{R}}:L\times \mathbb{R}\to \mathbb{R}$ are the canonical projections. Furthermore we have $B^{\varphi}=0$ and $J=0$, where $J$ is defined by \eqref{def di J}.
	\end{cor}
	\begin{proof}
		From \hyperref[thm che bach flat implica phi cotton flat]{Theorem \ref*{thm che bach flat implica phi cotton flat}} we obtain that $(M,g)$ is $\varphi$-Cotton flat. From \hyperref[cor se c phi zero allora rigidita]{Corollary \ref*{cor se c phi zero allora rigidita}} we deduce the isometry with $L\times \mathbb{R}^k$, for some $k$. Now, since $\varphi$ is harmonic, we have $\mathring{B}^{\varphi}=B^{\varphi}$. Notice that, from \hyperref[prop che descrive modello rigido]{Proposition \ref*{prop che descrive modello rigido}}, since $\lambda>0$, the only chance to have $B^{\varphi}=0$ is that $k=1$, and moreover $J=0$. Hence the proof is concluded.
	\end{proof}
	
	In this final Remark we motivate in which sense it is natural to replace the Bach tensor with the $\varphi$-Bach tensor.
	\begin{rmk}\label{rmk bach flat gen rel}
		In \cite{A20} we motivated the study of the pair of equations
		\begin{equation}\label{eq per funzionale}
			B^{\varphi}=0, \quad J=0
		\end{equation}
		on the compact four dimensional smooth manifold $M$, for a Riemannian metric $g$ and the smooth map $\varphi:M\to N$, where $(N,\langle\,,\,\rangle_N)$ is a fixed target Riemannian manifold. The solutions of \eqref{eq per funzionale} are characterized as critical points of the functional
		\begin{equation}\label{functional s 2 con bienergia}
			\mathcal{S}_2(g,\varphi):= \int_MS_2(A^{\varphi}_g)\mu_g-\frac{\alpha}{2}\int_M|\tau_g(\varphi)|^2\mu_g,
		\end{equation}
		where $\tau_g(\varphi)$ denotes the tension field of $\varphi$ evaluated with respect to the metric $g$, $S_2(A^{\varphi}_g)$ denotes the second elementary symmetric polynomial
		in the eigenvalues of the $\varphi$-Schouten tensor $A^{\varphi}_g$ of $(M,g)$ and $\mu_g$ is the Riemannian volume element of $(M,g)$.
		
		The functional \eqref{functional s 2 con bienergia} is the natural extension, in presence of the field $\varphi$, of the functional
		\begin{equation*}
			g\mapsto \int_M|W_g|^2_g\mu_g,
		\end{equation*}
		whose critical points in four dimension are characterized as Bach flat metrics.
		
		\hyperref[cor bach flat harm ricci soliton sono harm einst]{Corollary \ref*{cor bach flat harm ricci soliton sono harm einst}} shows that every $\varphi$-Bach flat compact four dimensional harmonic-Ricci soliton is harmonic-Einstein, i.e, critical point of the functional of normalized total $\varphi$-scalar curvature
		\begin{equation*}
			(g,\varphi)\mapsto \left(\int_M\mu_g\right)^{-\frac{1}{2}}\int_MS^{\varphi}_g\mu_g.
		\end{equation*}
		
		One may consider the equations \eqref{eq per funzionale} for a complete non-compact four dimensional manifolds. \hyperref[cor bach non comp flat harm ricci soliton sono harm einst]{Corollary \ref*{cor bach non comp flat harm ricci soliton sono harm einst}} shows that $\varphi$-Bach flat four dimensional complete non-compact gradient shrinking harmonic-Ricci solitons are isometric to a finite quotient of the Riemannian product of a three dimensional compact harmonic-Einstein manifold with the Gaussian shrinking soliton on $\mathbb{R}$. Furthermore, the validity of $B^{\varphi}=0$ implies automatically that $J=0$, exactly as in case $\varphi$ is a submersion a.e., see \cite{A20}.
	\end{rmk}


\begin{thebibliography}{}
		\bibitem[AMR]{AMR} L. J. Alias, P. Mastrolia, M. Rigoli - {\em Maximum principles and geometric applications}, Springer Monographs in Mathematics. Springer, Cham, 2016. xvii+570 pp. ISBN: 978-3-319-24335-1; 978-3-319-24337-5.
		\bibitem[A]{A} A. Anselli - {\em Phi-curvatures, harmonic-Einstein manifolds and Einstein-type structures}, PhD thesis, tutor: M. Rigoli; coordinatore: V. Mastropietro - Milano : Universit\'{a} degli studi di Milano. Dipartimento di matematica "Federigo Enriques", 2020 Jan 28. (32. ciclo, Anno Accademico 2019).
		\bibitem[A20]{A20} A. Anselli - {\em On the Bach and Einstein equations in presence of a field}, arXiv:2005.05943.
		\bibitem[ACR]{ACR} A. Anselli, G. Colombo, M. Rigoli - {\em On the geometry of Einstein-type structures}, preprint, DOI: 10.13140/RG.2.2.14896.51200.
		\bibitem[ABR]{ABR} C. Aquino, A. Barros, E. Jr. Ribeiro - {\em Some applications of the Hodge–de Rham decomposition to Ricci solitons}, Results Math. 60 (2011), no. 1-4, 245-254.
		\bibitem[B]{B} R. Bach, {\em Zur Weylschen Relativit\"{a}tstheorie und der Weylschen Erweiterung des Kr\"{u}mmungstensorbegriffs}, Mathematische Zeitschrift, 9 (1921), 110-135.
		\bibitem[BaE]{BaE} P. Baird, J. Eells - {\em A conservation law for harmonic maps}, Geometry Symposium, Utrecht 1980 (Utrecht, 1980), pp. 1-25, Lecture Notes in Math., 894, Springer, Berlin-New York, 1981.
		\bibitem[BW]{BW} P. Baird, J. C. Wood - {\em Harmonic morphisms between Riemannian manifolds}, London Mathematical Society Monographs. New Series, 29. The Clarendon Press, Oxford University Press, Oxford, 2003. xvi+520 pp. ISBN: 0-19-850362-8.
		\bibitem[C]{C} H.-D. Cao - {\em Recent progress on Ricci solitons}, Recent advances in geometric analysis, 1–38, Adv. Lect. Math. (ALM), 11 Int. Press, Somerville, MA, 2010.
		\bibitem[CC]{CC} H. D. Cao, Q. Chen - {\em On Bach-flat gradient shrinking Ricci solitons}, Duke Math. J. 162 (2013), no. 6, 1149–1169.
		\bibitem[dR]{dR} G. De Rham - {\em Sur la réductibilité d’un espace de Riemann}, Commentarii Mathematici Helvetici 26 (1952), 328-344.
		\bibitem[DMVVZ]{DMVVZ} {\em Geometry and topology of submanifolds. IX}. Dedicated to Professor Radu Rosca on the occasion of his 90th birthday. Papers from the Meetings on Geometry and Topology of Submanifolds held in Valenciennes, March 26–27, Lyon, May 17–18, and Leuven, September 18–20, 1997. Edited by F. Defever, J. M. Morvan, I. Van de Woestijne, L. Verstraelen and G. Zafindratafa. World Scientific Publishing Co., Inc., River Edge, NJ, 1999. x+236 pp. ISBN: 981-02-3897-5.
		\bibitem[FG]{FG} M. Fernández-López, E. García-Río - {\em Rigidity of shrinking Ricci solitons}, Math. Z. 269 (2011), no. 1-2, 461-466.
		\bibitem[ELM]{ELM} M. Eminenti, G. La Nave, C. Mantegazza - {\em Ricci solitons: the equation point of view}, Manuscripta Math. 127:3 (2008), 345-367.
		\bibitem[H]{H} R. S. Hamilton - {\em Three-manifolds with positive Ricci curvature}, J. Diff. Geom. 17 (1982), 255-306.
		\bibitem[J]{J} D. D. Joyce - {\em Riemannian holonomy groups and calibrated geometry}, Oxford Graduate Texts in Mathematics, 12. Oxford University Press, Oxford, 2007. x+303 pp. ISBN: 978-0-19-921559-1.
		\bibitem[L]{L} B. List - {\em Evolution of an extended Ricci flow system}, Comm Anal Geom, 2008, 16: 1007-1048.
		\bibitem[MS]{MS} O. Munteanu, N. Sesum - {\em On Gradient Ricci Solitons} - Journal of Geometric Analysis volume 23, pages 539-561 (2013).
		\bibitem[M]{M} R. M\"uller - {\em Ricci flow coupled with harmonic map flow}, Ann. Sci. \'Ec. Norm. Sup\'er. (4) 45 (2012), no. 1, 101-142.
		\bibitem[N]{N} L. Naber - {\em Some geometry and analysis on Ricci solitons}, arXiv:math.DG/0612532.
		\bibitem[PW]{PW} P. Petersen, W. Wylie - {\em Rigidity of gradient Ricci solitons}, Pacific J. Math. 241 (2009), no. 2, 329–345.
		\bibitem[T]{T} Y. Tashiro - {\em Complete Riemannian manifolds and some vector fields}, Trans. Amer. Math. Soc. 117 (1965), 251-275.
		\bibitem[W]{W} L. F. Wang - {\em On Ricci-harmonic metrics}, Ann. Acad. Sci. Fenn. Math. 41 (2016), no. 1, 417–437.
		\bibitem[YZ]{YZ} F. Yang, L.-D. Zhang - {\em On comlete Shrinking Ricci-harmonic solitons}, J. of Math. (PRC), Vol. 36 (2016), no. 3, 494-500.
		\bibitem[YS]{YS} F. Yang, J. Shen - {\em Volume growth for gradient shrinking solitons of Ricci-harmonic flow}, Sci. China Math. 55 (2012), no. 6, 1221-1228.
	\end{thebibliography}
\end{document}